\documentclass{article}

\usepackage[utf8]{inputenc}

\usepackage{color}
\usepackage{xcolor}
\usepackage{graphicx}
\usepackage{amsthm}
\usepackage{amsmath}
\usepackage{amsfonts}
\usepackage{array}
\usepackage{float}
\usepackage{enumitem}
\usepackage{caption} 
\usepackage{booktabs}
\usepackage{verbatim}
\usepackage{mathtools}
\usepackage{changepage}
\usepackage{hyperref}
\usepackage{url}

\hypersetup{
    colorlinks,
    linkcolor={red!50!black},
    citecolor={blue!50!black},
    urlcolor={blue!80!black}
}

\newtheorem{rem}{Remark}
\newtheorem{example}{Example}

\newcommand{\email}[1]{\protect\href{mailto:#1}{#1}}
\newcommand\funding[1]{\protect\\ \hspace*{15.37pt}{\bfseries Funding:} #1}

\newtheorem{definition}{Definition}
\newtheorem{theorem}{Theorem}
\newtheorem{proposition}{Proposition}
\newtheorem{lemma}{Lemma}
\newtheorem{corollary}{Corollary}

\newcommand{\Rn}{\mathbb{R}^n}

\newcommand{\norm}[1]{\left\lVert#1\right\rVert}
\newcommand{\ideal}[1]{\left\langle#1\right\rangle}
\newcommand{\R}{\mathbb{R}}
\newcommand{\C}{\mathbb{C}}
\newcommand{\Q}{\mathbb{Q}}
\newcommand{\Z}{\mathbb{Z}}

\newcommand{\T}{\operatorname{T}}
\newcommand{\vol}[1]{\operatorname{vol}(#1)}

\newcommand{\LL}{\mathcal{L}}
\newcommand{\LLL}{\mathcal{L}_{\mathbf{v}_1}^{(r,s)}}
\newcommand{\LLLL}{\mathcal{L}_{\T,\mathbf{v}_1}^{(r,s)}}

\newcommand{\set}[1]{\left\{#1\right\}}
\newcommand{\B}[1]{\mathbf{#1}}

\newcommand{\shrinkmargins}[1]{
  \addtolength{\textheight}{#1\topmargin}
  \addtolength{\textheight}{#1\topmargin}
  \addtolength{\textwidth}{#1\oddsidemargin}
  \addtolength{\textwidth}{#1\evensidemargin}
  \addtolength{\topmargin}{-#1\topmargin}
  \addtolength{\oddsidemargin}{-#1\oddsidemargin}
 \addtolength{\evensidemargin}{-#1\evensidemargin}
  }

\shrinkmargins{.7}

\title{{Dense Generic Well-Rounded Lattices}\thanks{Revision submitted to the editors \today. \funding{This work has been supported by the Research Council of Finland under Grant No. 351271 (PI~C.~ Hollanti). The visiting professor funding from the Aalto Science Institute to G.~Mantilla-Soler is also gratefully acknowledged.}}}

\author{Camilla Hollanti\footnotemark[2]\thanks{Department of Mathematics and Systems Analysis, Aalto University, Finland, \email{camilla.hollanti@aalto.fi}, \email{niklas.miller@aalto.fi}.}\and Guillermo Mantilla-Soler\footnotemark[3]\thanks{Department of Mathematics, Universidad Nacional de Colombia sede Medell\'in, Colombia, 
\email{gmantelia@gmail.com}.}\and
Niklas Miller\footnotemark[2]}

\begin{document}

\maketitle

\begin{abstract}
    It is well-known that the densest lattice sphere packings also typically have large kissing numbers. The sphere packing density maximization problem is known to have a solution among well-rounded lattices, of which the integer lattice $\Z^n$ is the simplest example. The integer lattice is also an example of a generic well-rounded lattice, \emph{i.e.}, a well-rounded lattice with a minimal kissing number. However, the integer lattice has the worst density among well-rounded lattices.   In this paper, the problem of constructing explicit  generic well-rounded  lattices with dense sphere packings is considered. To this end, so-called  \emph{tame} lattices recently introduced by Damir and Mantilla-Soler are utilized. Tame lattices came to be as a generalization of the ring of integers of certain abelian number fields. The sublattices of tame lattices constructed in this paper are shown to always result in either a generic well-rounded lattice or the lattice $A_n$, with  density ranging between that of $\Z^n$ and $A_n$. In order to find generic well-rounded lattices with densities beyond that of $A_n$, explicit deformations of some known densest lattice packings are constructed, yielding a family of generic well-rounded lattices with densities arbitrarily close to the optimum. In addition to being an interesting mathematical problem on its own right, the constructions are also motivated from a more practical point of view. Namely, generic well-rounded lattices with high packing density make good candidates for lattice codes used in secure wireless communications.
\end{abstract}

\begin{adjustwidth}{25pt}{25pt}
\end{adjustwidth}

\begin{adjustwidth}{25pt}{25pt}
\end{adjustwidth}






\section{Introduction}

Lattices are of broad interest both in mathematics and applications. \emph{Well-rounded (WR) lattices} \cite{Fukshansky1,Fukshansky2} are a profound family of lattices, providing a sufficient source for various optimization problems. For instance, the known densest lattice sphere packings (\emph{e.g.}, the Gosset lattice $E_8$) arise from well-rounded lattices, and it has been shown that the lattice sphere packing problem can be restricted to the space of well-rounded lattices also in general. A prominent feature of well-rounded lattices is that the set of shortest vectors span the ambient space. 
Previously, constructions of well-rounded lattices have been considered in, \emph{e.g.},  \cite{Taoufiq-Dave,Alves,costawell,bottcher2015lattices,bacher2015constructions}, and they also relate to the famous Minkowski and Woods conjectures \cite{McMullenMinkowski}.

One way to describe a lattice $\Lambda$ is via its theta series (see Section \ref{sec:preli} for a rigorous definition)
$$\Theta_\Lambda(q):=\sum_{\B{x}\in\Lambda} q^{\norm{\B{x}}^2},\quad |q|<1\,,$$
which is also closely related to the so-called \emph{flatness factor} utilized in wireless security \cite{Belfiore-flatness} and \emph{smoothing parameter} in cryptography \cite{Regev-smoothing}. In \cite{towards}, it was shown that well-rounded lattices yield promising candidates for lattice codes in wireless communications \cite{viterbo} and physical layer security \cite{Chorti}. This stems from the fact that certain optimization problems related to such code design can be restricted to the set of well-rounded lattices. This connection will be elaborated more in Section \ref{sec:app}. Here, let us just briefly mention that one such optimization problem is the minimization of the theta function \cite{Belfiore2010Secrecy-gain,towards}. However, computing the value of a theta function is far from easy, which  makes finding the minimum in general very hard. Indeed, the minimum is known only in few dimensions, and is not always the densest lattice packing \cite{Montgomery}.  To this end, if we look at a simple truncation of the theta function corresponding to the first minimum, it becomes evident that the minimization problem in this case can be formulated as simultaneously maximizing the first minimum (\emph{i.e.}, the lattice packing density) and minimizing the number of shortest vectors  (\emph{i.e.}, the kissing number). Unfortunately, there is no obvious tradeoff between these two goals. Namely, the densest packings tend to exhibit large kissing numbers. Hence, it is interesting to study dense \emph{generic well-rounded (GWR) lattices} having  minimal kissing numbers, as proposed in \cite{towards}. Naturally, the truncation approximation is coarse. In \cite{Amaro_approx}, a novel theta function approximation was derived, enabling efficient comparison of candidate lattice codes in the above light. 


Motivated by these findings, \emph{tame lattices} were introduced in \cite{damir2020bases}, providing a source for constructing explicit well-rounded lattices. In the present paper, we take a step further and show that tame lattices also give rise to generic well-rounded lattices. Moreover, we show that by distorting some of the densest lattice packings, it is possible to construct lattices that simultaneously have a large minimum distance and a small kissing number. Such lattices should provide ideal solutions for reliable and secure communications. More specifically, our main contributions are the following:

\begin{itemize}
    \item In Theorem \ref{thm:l2}, we state a stronger version of Theorem 4.9 in \cite{damir2020bases}, namely that a well-rounded lattice of the form $\LLL$ (see Definition \ref{def:l1}) is GWR or similar to $A_n$. 
    \item In Proposition \ref{prop:l2}, we prove that for the center densities it holds that  $$\delta(\Z^n)\leq\delta(\LLL)\leq\delta(A_n).$$
    \item In Theorem \ref{thm:4}, we give two conditions which tell when a well-rounded lattice $\LLL$ is similar to either $\Z^n$ or $A_n$. This is particularly interesting from the applications point-of-view, as discussed in Section  \ref{sec:conclusion}.
    \item In Theorem \ref{DualOfLinearLattices}, we show that the dual of a lattice of the form $\LLLL$ is of the same type.
    \item In Sections \ref{d_n} and \ref{e_8} we introduce \emph{deformed} lattices $D_n^\alpha$ and $E_8^\alpha$, where $\alpha$ is a real parameter, and in Theorems \ref{Dn_deformed:gwr} and \ref{conj} we prove that these lattices are GWR if they are not similar to $D_n$ and $E_8$, respectively. It is also shown that the center density can be brought arbitrarily close to the optimum. 
\end{itemize}

Let us next provide some mathematical motivation for the concept of tame lattices.


Let $p$ be an odd prime and let $K$ be a number field that is Galois over $\Q$ and such that ${\rm Gal}(K/\Q)\cong \Z/p\Z$. In  \cite[IV.8]{cp} Conner and Perlis showed that if $K$ is tame, \emph{i.e.}, no rational prime has wild ramification, 
then there exists an integral basis for $O_{K}$, which they called \emph{Lagrangian basis}, such that the Gram matrix of the integral trace in such basis is a $p \times p$ matrix of the form 
\[\left[
  \begin{array}{cccc}
    a & -h & \ldots & -h \\
   -h & \ a & \ddots & \vdots \\
    \vdots & \ddots & \ddots & -h \\
    -h & \ldots & -h & \ a \\
  \end{array}
\right]\,.\]
The values of $a$ and $h$ are  given by \[a=\frac{d(p-1)+1}{p}, \ \ h=\frac{d-1}{p},\] where $d$ is the conductor of $K$, \emph{i.e.}, the smallest positive integer such that $K$ is contained in the cyclotomic extension $\Q(\zeta_{d})$. Since $K$ is tame, $d$ is also equal to the product of the primes that ramify in $K$. Based on this explicit description of the trace form over $O_{K}$, in \cite{costawell} the authors have constructed families of well-rounded lattices that are sublattices of $O_{K}$. In \cite{damir2020bases}, the notion of tame lattice is introduced  by axiomatizing the key properties of the integral trace over degree $p$ Galois extensions. For instance notice that $a$ and $h$ satisfy $a -h(p-1)=1$ and $a,h \ge 0$. In \cite{damir2020bases} these lattices are used to  develop a procedure to construct well-rounded lattices, in fact strongly well-rounded, for which the previous results on Galois degree $p$ number fields case is just a particular example. In the present paper we explore the construction of tame lattices further. For instance, one of the key elements in the construction of strongly well-rounded lattices in \cite{damir2020bases} is the construction of a family of sublattices of the form \[\LL_{\T_{\B{v}_1},\B{v}_1}^{(r,s)},\]  for a given lattice $\LL$, see Definition \ref{LinearSubLattice} for details. Such construction is inspired by the study of the trace form over real number fields but its applications go beyond such fields. For instance we  show that several lattices, including 
$A_n, D_n, E_8$ and their duals, all arise as special cases of this construction. Furthermore we study the behavior of quantities such as the center density over such families, and exhibit examples of lattices  with large center density.  

\subsection{Organization}
The rest of the paper is organized as follows. In Section \ref{sec:preli}, we introduce the mathematical preliminaries needed in later sections and discuss application potential in secure wireless communications. In Section \ref{sec:tame} we study well-rounded sublattices of tame lattices. We give bounds for the center densities of such lattices and characterize them under certain conditions. In Section \ref{sec:gwr} we define deformed lattices $E_8^\alpha$ and $D_n^\alpha$, which are families of GWR lattices with density approaching to the optimum as $\alpha\to1$. In Section \ref{sec:conclusion}, we conclude this work and give directions for future research.

\section{Preliminaries}
\label{sec:preli}

In this section we define some basic concepts  that we will need later on, as well as discuss some connections to applications. Vectors will always be column vectors, and will be denoted by boldface letters. 

\subsection{Lattices}

A standard reference for the theory of lattices is the book \cite{conway}.

\begin{definition}
A lattice $\Lambda$ is a discrete subgroup of the additive group $(\Rn,+)$. Equivalently, it is a set 
$$\Lambda=\set{\sum_{i=1}^m u_i \B{b}_i : u_i\in \Z},$$ where $\set{\B{b}_1, \B{b}_2,\dots, \B{b}_m}$ is a set of linearly independent vectors in $\R^n$ called a basis of $\Lambda$, and $m$ is called the rank of $\Lambda$. If $m=n$, then $\Lambda$ is said to be full rank.
\end{definition}

All lattices considered in this text are full rank. The non-singular square matrix $M$ whose columns consist of the basis vectors $\B{b}_i$, $1\leq i\leq n$, is called a \emph{generator matrix} for the lattice $\Lambda$. The generator matrix is not unique: if $M$ is a generator matrix for $\Lambda$ and $U\in\Z^{n\times n}$ is a unimodular matrix, then $M'=MU$ is another generator matrix for $\Lambda$. We have the following two important examples of lattices frequently occuring in this article.

\begin{example}
  The orthogonal lattice $\Z^n$ is a lattice with a generator matrix $I_n$, the identity matrix.
\end{example}
\begin{example}
 The root lattice $A_n$ is a rank $n$ lattice in $\R^{n+1}$ defined by 
 $$A_n:=\set{(x_1,\dots,x_{n+1})^T\in\Z^{n+1}: \sum_{i=1}^{n+1}x_i=0}.$$ Basis vectors are given by $\B{v}_i=\B{e}_i-\B{e}_{i+1}$, for $1\leq i\leq n$, where the $\B{e}_i$ are standard basis vectors in $\R^{n+1}$. The lattice $A_n$ can be embedded as a full rank lattice in $\mathbb{R}^n$ as follows: take the cyclic shifts of the vector $\mathbf{v}(x,y)=(x,y,y,\cdots,y)\in \mathbb{R}^n$ where $x=\frac{\sqrt{n+1}+(n-1)}{n}$ and $y=\frac{\sqrt{n+1}-1}{n}$. Then it is easy to check that the Gram matrix (defined below) of this lattice is of the form $I_n+J_n$ where $I_n$ is the identity matrix and $J_n$ is the matrix of all ones. This is the Gram matrix of the $A_n$ lattice \cite{conway}.
\end{example}

Information about the inner products between basis vectors of a lattice is encoded in the \emph{Gram matrix}, which we define to be the symmetric and positive definite matrix $$G:=M^TM=(\B{b}_i^T \B{b}_j)_{1\leq i,j\leq n}.$$

Since the generator matrix of a lattice is not unique, neither is the Gram matrix. Indeed, if $M$ is a generator matrix for $\Lambda$ and $M'=MU$ is another generator matrix for some unimodular $U$, then the corresponding Gram matrices are related by $G'=M'^TM'=(MU)^TMU=U^TGU$. It is however the case that $\det(G')=\det(G)$, motivating the definition of the \emph{determinant} of a lattice, which we simply denote by $\det(\Lambda):=\det(G)$, where $G$ is any Gram matrix of $\Lambda$. We define the \emph{volume} (sometimes known as the \emph{co-volume}) of a lattice $\Lambda$ to be the Lebesgue measure of a fundamental domain for the translation action of $\Lambda$ on $\mathbb{R}^n$, or simply $\vol{\Lambda}:=\sqrt{\det(\Lambda)}$. For a full rank lattice this reduces to $\vol{\Lambda}=|\det(M)|$. If a lattice $\Lambda$ has a Gram matrix $G$ with integral entries, then the lattice is said to be \emph{integral}.

We are interested in certain \emph{sublattices} of tame lattices. What we mean by a sublattice of a lattice $\Lambda$ is a lattice $\Lambda'\subseteq\Lambda$. If $\Lambda'$ and $\Lambda$ have the same rank, then we have the following formula for the index of $\Lambda'$ in $\Lambda$: $$[\Lambda:\Lambda']=\frac{\vol{\Lambda'}}{\vol{\Lambda}}.$$

Two lattices $\Lambda$ and $\Lambda'$ with generator matrices $M$ and $M'$, respectively, are said to be \emph{similar}, denoted $\Lambda\sim \Lambda'$, if one is obtained from the other by a rotation and dilation: $M'=\alpha BMU$ for some $\alpha>0$, a real orthogonal matrix $B$ and a unimodular matrix $U\in\Z^{n\times n}$. If $\alpha=1$, we say that $\Lambda$ and $\Lambda'$ are \emph{isometric} and denote it by $\Lambda\cong\Lambda'$. Hence, two lattices are similar if and only if one is obtained from the other by rotation, reflection and scaling. The Gram matrices of similar lattices $\Lambda$ and $\Lambda'$, as above, are related by $G'=\alpha^2 U^TG U$. Consequently, the volumes of the lattices are related by $\vol{\Lambda'}=\alpha^n\vol{\Lambda}$.

The \emph{dual lattice} of $\Lambda\subseteq \mathbb{R}^n$ is the dual module $\text{Hom}_{\mathbb{Z}}(\Lambda,\mathbb{Z})$, which we view as a lattice in $\R^n$ as follows.

\begin{definition}
Let $\Lambda\subset\Rn$ be a lattice. Its dual lattice $\Lambda^*$ is defined as
$$\Lambda^*:=\set{\B{v}\in\Rn : \ideal{\B{v},\B{x}}\in\Z \text{ for any }\B{x}\in\Lambda}.$$
\end{definition}

\begin{rem}
\label{rem:dual}
    Given a basis $\{\mathbf{b}_1,\dots,\mathbf{b}_n\}$ of a lattice $\Lambda$, the dual lattice $\Lambda^*$ is generated by $\{\mathbf{b}_1^*,\dots, \mathbf{b}_n^*\}$ where $\mathbf{b}_i^*\cdot \mathbf{b}_j=\delta_{ij}$ for all $1\leq i,j\leq n$, where $\delta_{ij}$ denotes the Kronecker delta symbol. Thus, it is clear that the dual $\Lambda^*$ is generated by the matrix $(M^T)^{-1}$, where $M$ is any generator matrix for $\Lambda$. A Gram matrix for $\Lambda^*$ is then given by $G^{-1}$, where $G$ is any Gram matrix of $\Lambda$.
\end{rem}



An important lattice invariant is the \emph{shortest vector length}, or \emph{first minimum}. Given a lattice $\Lambda\subset\R^n$, we define it to be the norm of the shortest non-zero lattice vector, \emph{i.e.}, $$\lambda_1(\Lambda):=\min_{0\neq \B{x}\in \Lambda}\norm{\B{x}}.$$

Maximizing $\lambda_1(\Lambda)$ over the set of rank $n$ lattices $\Lambda\subset\R^n$ of volume 1 is in general a hard problem: this problem is equivalent to finding the densest lattice packing of spheres, and the solution is known only in dimensions 1--8 and 24. A commonly used measure of sphere packing density of a lattice packing is the \emph{center density}, which we define to be
$$\delta(\Lambda):=\frac{\lambda_1(\Lambda)^n}{2^n\vol{\Lambda}}.$$
The center density is invariant under lattice similarity. We can see from the very definition that $\delta(\Lambda)=2^{-n}\lambda_1(\Lambda)^n$ for a volume 1 lattice $\Lambda$, showing that maximizing the center density is equivalent to maximizing the shortest vector length. We will denote by $\delta_n$ the largest known center density of a lattice in dimension $n$. The set $S(\Lambda)=\set{\B{x}\in \Lambda: \norm{\B{x}}=\lambda_1(\Lambda)}$ is called the \emph{set of minimal vectors} of $\Lambda$ and its cardinality $\kappa(\Lambda):=|S(\Lambda)|$ is called  the \emph{kissing number} of the lattice.

We will be concerned with specific types of lattices called \emph{well-rounded} lattices.

\begin{definition}
A lattice $\Lambda\subset \Rn$
\begin{enumerate}[label={(\roman*)}]
  \item\label{wr} is called \emph{well-rounded (WR)} if the $\R$-span of $S(\Lambda)$ is $\Rn$.
  \item\label{swr} is called \emph{strongly well-rounded (SWR)} if the $\Z$-span of $S(\Lambda)$ is $\Lambda$.
  \item\label{mb} has a \emph{basis of minimal vectors} if there exists a set $\set{\B{v}_1,\dots,\B{v}_n}\subseteq S(\Lambda)$ whose $\Z$-span is $\Lambda$.
\item is called \emph{generic well-rounded (GWR)}  if it is well-rounded and has kissing number $\kappa(\Lambda)=2n$. 
\end{enumerate}
\end{definition}
A common example of a GWR lattice is $\Z^n$.

The \emph{theta series} of a lattice is defined as 
\begin{equation}\label{eq:theta}
\Theta_\Lambda(q):=\sum_{\B{x}\in\Lambda} q^{\norm{\B{x}}^2},
\end{equation}
where $q=e^{i\pi\tau}$, and $\Im{\tau}>0$. We are interested in the case $\tau=it$, \emph{i.e.} when $q$ is real. We can express the theta function equivalently as 
$$\Theta_\Lambda(q)=1+k_1q^{l_1^2}+k_2q^{l_2^2}+\dots,$$
where $l_i$, $i=1,2,\dots$ is the length of the $i$:th shortest vector in the lattice and $k_i:=|\set{\B{x}\in\Lambda: \norm{\B{x}}=l_i}|$. In particular, $l_1=\lambda_1(\Lambda)$ and $k_1=\kappa(\Lambda)$.

\subsection{Secure wireless communications}\label{sec:app}
In this section, we will provide a short practical motivation for the construction of generic well-rounded lattices with large packing densities. We refer the interested reader to \cite{viterbo,belfiore,Chorti,costa2017lattices} and references therein for more details on reliable and secure wireless communications. 

When communicating over a wireless channel,  reliability and security of the transmission are natural concerns. Lattice codes have manifested themselves as a good design tool for this purpose. A \emph{lattice code} is a finite collection of lattice vectors centered at the origin, often carved out by a hypersphere or a hypercube. The related design criteria may vary depending on the specific application and channel model. Over the wireless medium, the channel model can be described as
$$
\mathbf{y}=H\mathbf{x}+\mathbf{n}\in \Rn,
$$
where $\mathbf{x}\in \Lambda$ is the transmitted message, $\mathbf{n}\in \Rn$ is additive white Gaussian noise (AWGN), $H\in \R^{n\times n}$ is random fading following a certain (real, continuous) distribution, and $\mathbf{y}$ is the received distorted message. In the case of no fading, the channel is simply called an \emph{AWGN channel} and $H=I_n$. Another typical model is the \emph{Rayleigh fading channel}, which refers to the case where $H$ is diagonal with Rayleigh distributed independent entries. Maximum-likelihood (ML) decoding with this model corresponds to (bounded distance) closest lattice point search.

If the communication channel can be expected to be of a good quality, measured by \emph{signal-to-noise ratio (SNR)}, and we assume a Rayleigh fading channel, then reliability (under ML decoding) can be maximized by maximizing the \emph{modulation diversity} $L$ and \emph{minimum product distance} of the lattice \cite{viterbo}. The former is defined as $$L:=\min_{0\neq\B{x}\in\Lambda}|\set{i: x_i\neq0}|,$$ the minimum number of non-zero coordinates in a non-zero lattice vector, and the latter is defined as $$d_{p,min}(\Lambda):=\inf_{0\neq\B{x}\in\Lambda}\prod_{i=1}^{n}|x_i|$$ for a full diversity lattice with $L=n$. For a low signal quality, the minimum distance $\lambda_1(\Lambda)$  becomes crucial.

When considering security in addition to reliability, we choose a sublattice $\Lambda_s\subset \Lambda$ and the respective cosets represent our messages. In more detail, the lattice vector  $\mathbf{m}\in \Lambda$ is ``masked'' by a randomly chosen sublattice vector $r$, and the message  $\mathbf{x}=\mathbf{m}+\mathbf{r}\in \Lambda/\Lambda_s$ is sent through the channel. On this type of a \emph{wiretap channel} \cite{wyner, Wyner-Ozarow} we assume that the noise experienced by the eavesdropper is larger than that of the legitimate receiver. As a result, the legitimate receiver is able to decode the message with high probability, while the eavesdropper only gains negligible information. For more details on wiretap coset coding and illustrating examples, see \cite{Oggier-Sole-Belfiore}.

The security of the aforementioned wiretap lattice coding can be measured by the information leaked to the eavesdropper or alternatively by the eavesdropper's correct decoding probability. Both of these measures have been shown \cite{Belfiore-flatness,luzzi_isit16,towards}  to be bounded from above by the \emph{flatness factor}, which is defined by the deviation of the lattice Gaussian  probability density function $g_n$ from the uniform distribution on the Voronoi cell  $\mathcal{V}(\Lambda)$.  More rigorously, the flatness factor $\varepsilon_{\Lambda} (\sigma)$  is characterized by
\begin{equation}
\label{eq: flatness factor definition}
\varepsilon_{\Lambda} (\sigma) : = \max_{\mathbf{u} \in \R^n} \left| \frac{g_{n}(\Lambda + \mathbf{u}; \sigma)}{1/\vol{\Lambda}} - 1 \right|,
\end{equation}
where we can maximize over $\R^n$ by periodicity. One can think of this as the measure of ``how much better'' an eavesdropper can do than just random guessing when trying to intercept (decode) the sent message. The ``flatter'' the channel, the closer it is to a uniform observation from the eavesdropper's perspective. Therefore, a reasonable design goal is to minimize the flatness factor. 
We have the following useful equalities:
\begin{equation}
\label{eq: dual theta flatness factor formula}
\varepsilon_{\Lambda} (\sigma) 
=  \vol{\Lambda} g_{n} (\Lambda; \sigma) - 1 
= \frac{\vol{\Lambda}}{(\sqrt{2 \pi} \sigma)^n} \Theta_{\Lambda} (e^{-1/2 \sigma^2}) - 1 
= \Theta_{\Lambda^*} (e^{- 2 \pi \sigma^2}) - 1.
\end{equation}
Here, $\sigma^2=\frac{1}{2\pi t}$ is the variance of the Gaussian distribution related to the channel noise. 

It was proved in \cite{towards} that the minimizer of the flatness factor is a WR lattice. 
Now, the benefit from WR lattices is now multi-fold: they contain the maximizers of the minimum product distance, the densest lattice sphere packings, as well as the global minimizers of the flatness factor, which is closely related to the theta series as shown by Equation \eqref{eq: dual theta flatness factor formula}. For more details, see \cite[Sec. VI]{towards},  \cite{Sarnak-Strombergsson,coulangeon,delone}. Hence, well-rounded lattices make an excellent candidate space for reliable and secure wireless communications.

As we mentioned, the minimization of the theta function is in general a very hard problem. Hence, we can aim at studying its truncations. The longer the shortest vector and the smaller the kissing number, the smaller the dominating term in the theta series becomes (see Section \ref{sec:gwr} for more details). From this,  the question of how to construct GWR lattices with good packing densities has a clear practical motivation and acts as the driving force for the findings of this paper.

\section{Tame lattices}
\label{sec:tame}

In this section, we define the concept of a \emph{tame lattice}, which were introduced in \cite{damir2020bases}, and expand on the results in \cite{damir2020bases} concerning well-rounded sublattices of tame lattices. We give conditions under which well-rounded sublattices of tame lattices are generic, and develop bounds for the center densities of said sublattices. We end up with two conditions which tell when a sublattice is similar to either $\Z^n$ or $A_n$. We also give an explicit construction of $A_n$ as a sublattice of a tame lattice, and further, show that the dual lattice of a tame lattice is tame.

The term \emph{Lagrangian basis} was used by Conner and Perlis \cite[p. 193]{cp} to describe a normal integral basis of a Galois, tame number field of odd prime degree, such that the Gram matrix of the lattice obtained from this basis via the Minkowski embedding has a certain form. Later, it has been proven that such bases exist in a larger set of number fields. Particularly, when the number field is tame and has a prime conductor, such a basis has been shown to exist \cite{bm}. The following definition is given in \cite{damir2020bases}.

\begin{definition}
\label{defn:l1}
Let $n \geq1$ be an integer and $\LL\subset \R^n$ a lattice. We call the lattice $\LL$ tame if there exists a basis $\set{\B{e}_1,\dots,\B{e}_n}$ of $\LL$ and a non-zero vector $\B{v}_1\in\LL\cap\LL^*$ such that
\begin{enumerate}
  \item $\sum_{i=1}^n \B{e}_i=\B{v}_1$
  \item $\ideal{\B{e}_i,\B{v}_1}=1 \text{ for all }1\leq i\leq n$
  \item $\ideal{\B{e}_i,\B{e}_i}=a \text{ for all }1\leq i\leq n$
  \item $\ideal{\B{e}_i,\B{e}_j}=-h \text{ for all } 1\leq i\neq j\leq n.$
\end{enumerate} 
In this case, we call $\set{\B{e}_1,\dots,\B{e}_n}$ a Lagrangian basis for $\LL$.
\end{definition}


Conditions 3 and 4 give an explicit expression for the Gram matrix of a tame lattice with respect to the basis $\set{\B{e}_1,\dots,\B{e}_n}$:
\begin{equation}
\label{lag:gram}
G= \left[
  \begin{array}{cccc}
    a & -h & \ldots & -h \\
   -h & a & \ddots & \vdots \\
    \vdots & \ddots & \ddots & -h \\
    -h & \ldots & -h & a \\
  \end{array}
\right].\end{equation}

Moreover, the conditions imply that $a-h(n-1)=1$. Conversely, a basis with a Gram matrix of the form (\ref{lag:gram}) such that $a-h(n-1)=1$ holds, is Lagrangian. We can express the volume of a tame lattice using (\ref{lag:gram}).

\begin{lemma}
	\label{lem:l2}
	Let $\LL\subset \R^n$ be a tame lattice with a Lagrangian basis $\{\B{e}_1,\dots,\B{e}_{n}\}$, $a:=\ideal{\B{e}_1,\B{e}_1}$ and $h:=-\ideal{\B{e}_1,\B{e}_2}$. Then 
	$$\vol\LL=(a+h)^{\frac{n-1}{2}}.$$
\end{lemma}

\begin{proof} 

Consider the matrix $M_n(x,y)=xI_n+yJ_n$ where $x,y\in\mathbb{R}$ and $J_n$ is the matrix of all ones. If $\mathbf{v}\in\mathbb{R}^n\setminus\{0\}$ is an eigenvector for $M_n(x,y)$ with eigenvalue $\lambda$, then $M_n(x,y)\mathbf{v}=x\mathbf{v}+y \langle \mathbf{v},\mathbf{u}\rangle \mathbf{u}=\lambda\mathbf{v}$, where $\mathbf{u}$ is the vector of all ones in the standard basis. This shows that either $\langle \mathbf{v},\mathbf{u}\rangle=0$ and $\lambda=x$ is an eigenvalue with multiplicity $n-1$ or $\mathbf{v}$ is parallel to $\mathbf{u}$ and $\lambda=x+yn$ is an eigenvalue with multiplicity one. Therefore, the determinant equals $\det(M_n(x,y))=(x+yn)x^{n-1}$. If we plug in $x=a+h$ and $y=-h$ we obtain the determinant of (\ref{lag:gram}) as
$$\det(G)=(a+h)^{n-1}(a+h(1-n)).$$
Since $a+h(1-n)=1$,
$$\vol\LL=\sqrt{\det(G)}=\sqrt{(a+h)^{n-1}}=(a+h)^{\frac{n-1}{2}},$$
as desired. \end{proof}

\begin{rem}
	\label{rem:1}
	If $\LL\subset\R^n$ is a tame lattice with a Gram matrix $G$, then since $G$ is a positive definite matrix, we have $\det(G)=(a+h)^{n-1}>0$ and in particular, the inequality $a+h>0$ and equation $a+h(1-n)=1$ imply $a>\frac{1}{n}$ and $h>-\frac{1}{n}$.
\end{rem}

\subsection{Well-rounded sublattices}

In this section, we present some of the definitions and results in \cite{damir2020bases}, since our goal is to expand on these results. We recall how one can produce full-rank sublattices, and particularly tame lattices, via a specific linear map. We also restate the main theorem presented in \cite{damir2020bases}.

\begin{rem}\label{Motivation def r,s,T}
The motivation behind the following definition of $\LLLL:=\Phi_{(r,s)}(\LL)$ comes from number theory. Let $p$ be an odd prime and let $K$ be a Galois degree $p$ number field, where $p$ is unramified. 
For an integer $m \equiv 1 \pmod{p}$ it turns out the ring of integers $O_{K}$, seen as a lattice via the Minkowski embedding, contains a sub-lattice $O_{m}$ that has a minimal basis. Such sub-lattice is a particular instance of $\LLLL$; in this case $\LL$ is $O_{K}$, $\B{v}_1=1$, $T$ is the trace map, $r=1$ and $s=\frac{m-1}{p}$. For details, see Section 1.1 and Lemma 3.8 in 
\cite{damir2020bases}. 
\end{rem}

\begin{definition}\label{LinearSubLattice}
\label{def:l1}
Let $\LL\subset \R^n$ be a lattice and $\T:\LL\rightarrow\Z$ a non-trivial linear map. Let $r,s$ be integers, $\B{v}_1\in \LL\setminus\ker \T$ and $m:=r+s\T(\B{v}_1)$. Define $\Phi_{(r,s)}:\LL\rightarrow\LL$ to be the linear map
$$ \B{x}\mapsto r\B{x}+s\T(\B{x})\B{v}_1.$$
Define the lattice $\LLLL$ to be the image of $\LL$ under the map $\Phi_{(r,s)}$:
$$\LLLL:=\Phi_{(r,s)}(\LL).$$
\end{definition}

\begin{rem}\label{Condition r m} It is clear that $\LLLL$ is a sublattice of $\LL$. Further, if $r$ and $m$ are non-zero, then $\Phi_{(r,s)}$ turns out to be an injection. This can be achieved for instance when $0<|r|<|\T(\B{v}_1)|$, whence we get the following result (Corollary 3.2 in \cite{damir2020bases}).
\end{rem}

\begin{lemma}
\label{lem:l1}
Assume that $\LL$, $\T$ and $\B{v}_1$ are as in Definition \ref{def:l1} and $r,s$ are integers such that $0<|r|<|\T(\B{v}_1)|$. Then $\LL_{\T,\B{v}_1}^{(r,s)}$ is a full rank sublattice of $\LL$.
\end{lemma}

If the conditions in the above lemma are satisfied and if $\{\B{e}_1,\dots,\B{e}_n\}$ is a basis for $\LL$, then $\set{\Phi_{(r,s)}(\B{e}_1),\dots,\Phi_{(r,s)}(\B{e}_n)}$ is a basis for $\LLLL$. We will be particularly interested in the following linear map:
$$\T_{\B{v}_1}:\LL\rightarrow\Z,\quad \T_{\B{v}_1}(\B{x})=\ideal{\B{x},\B{v}_1},$$
where $0\neq \B{v}_1\in\LL\cap\LL^*$. We define $$\LL_{\B{v}_1}^{(r,s)}:=\LL_{\T_{\B{v}_1},\B{v}_1}^{(r,s)}.$$

In \cite{damir2020bases}, the authors give a condition which tells when $\LLL$ is well-rounded, and in fact, possesses a minimal basis. If the condition is met, then the index $[\LLL:\LL]$ and the lattice minimum of $\LLL$ are given. More specifically:

\begin{theorem}
\label{thm:l1}\cite[Theorem 4.9]{damir2020bases}
Let $n\geq2$ be an integer and $\mathcal{L}\subset\R^n$ a tame lattice with a Lagrangian basis $\set{\B{e}_1,\dots,\B{e}_n}$. Let $a:=\ideal{\B{e}_1,\B{e}_1}$ and $h:=-\ideal{\B{e}_1,\B{e}_2}$. Let $r,s$ be integers such that $0\neq |r|<n$ and let $m=r+sn$. Suppose that 
\begin{equation}
\label{bounds}
\frac{na-1}{n^2-1}\leq \left(\frac{m}{r}\right)^2\leq\frac{(na-1)(n+1)}{n-1}.
\end{equation}
Then $\LL_{\B{v}_1}^{(r,s)}$ is a full rank sublattice of $\LL$ of index $|mr^{n-1}|$, with minimum 
$$\lambda_1^2(\LL_{\B{v}_1}^{(r,s)})=ar^2+\frac{m^2-r^2}{n}$$
and a basis of minimal vectors $\{r\B{e}_1+s\B{v}_1,\dots,r\B{e}_n+s\B{v}_1 \}$.
\end{theorem}

In fact a stronger result is true: if the upper bound in (\ref{bounds}) is strict, then $\LLL$ is GWR. Namely, in this case, \cite[Cor. 4.8]{damir2020bases} shows that the vectors $\Phi_{(r,s)}(\B{e}_i)$ have strictly smaller length than the vectors coming from the sublattice $\mathcal{L}^0$ of vectors of trace zero (since the upper bound is equivalent to $Aa+B\leq 2A(a+h)$). Moreover, \cite[Prop. 4.7]{damir2020bases}  shows that the vectors $\Phi_{(r,s)}(\B{e}_i)$ also have smaller length than the vectors $\Phi_{(r,s)}(\B{v})$, where $\B{v}\in S_d=\{\B{x}\in \mathcal{L}:T(\B{x})=d,\, d\neq 0\}$.


\begin{theorem}
\label{thm:l2}
Let $n\geq 2$ be an integer and let $\mathcal{L}\subset\R^n$ be a tame lattice with a Lagrangian basis $\set{\B{e}_1,\dots,\B{e}_n}$. Let $a:=\ideal{\B{e}_1,\B{e}_1}$ and $h:=-\ideal{\B{e}_1,\B{e}_2}$. Let $r,s$ be integers such that $0\neq |r|<n$ and let $m=r+sn$. Suppose that 
$$\frac{na-1}{n^2-1}\leq \left(\frac{m}{r}\right)^2<\frac{(na-1)(n+1)}{n-1}.$$
Then $\LL_{\B{v}_1}^{(r,s)}$ is GWR.
\end{theorem}


An interesting observation is that, when the upper bound in (\ref{bounds}) is achieved, then $\LLL$ has the same kissing number as $A_n$, which is $n(n+1)$. In particular, $\LLL$ is no longer GWR. Later (in Theorem \ref{thm:4}) we will show that $\LLL$ is actually similar to $A_n$ in this case.


\subsection{Center densities of sublattices of tame lattices}

Our goal in this section is to characterize the center densities of the lattices $\LLL$ in \ref{def:l1} and develop bounds for them. The following proposition gives an explicit formula for the center density.

\begin{proposition}
\label{prop:l1}
Let $n\geq 2$ be an integer and let $\mathcal{L}\subset\R^n$ be a tame lattice with a Lagrangian basis $\set{\B{e}_1,\dots,\B{e}_n}$. Let $a:=\ideal{\B{e}_1,\B{e}_1}$ and $h:=-\ideal{\B{e}_1,\B{e}_2}$. Let $r,s$ be integers such that $0\neq |r|<n$ and let $m=r+sn$. Suppose that 
$$\frac{na-1}{n^2-1}\leq \left(\frac{m}{r}\right)^2\leq\frac{(na-1)(n+1)}{n-1}.$$ Then the center density of $\LL_{\B{v}_1}^{(r,s)}$ is given by
$$\delta(\LL_{\B{v}_1}^{(r,s)})=\frac{((na-1)r^2+m^2)^{n/2}}{2^n n^{n/2} (a+h)^{\frac{n-1}{2}}  |mr^{n-1}|} .$$
\end{proposition}

\begin{proof} \ Note that by Lemma \ref{lem:l2} and Theorem \ref{thm:l1},  $$\vol{\LL_{\B{v}_1}^{(r,s)}}=\vol{\LL}[\LL:\LL_{\B{v}_1}^{(r,s)}]=(a+h)^{\frac{n-1}{2}}|mr^{n-1}|.$$ Using the definition of center density and the expression for the shortest vector length of $\LL_{\B{v}_1}^{(r,s)}$, we get
\begin{align*}
\delta(\LL_{\B{v}_1}^{(r,s)})&=\frac{\lambda_1^2(\LL_{\B{v}_1}^{(r,s)})^{n/2}}{2^n \vol{\LL_{\B{v}_1}^{(r,s)}}}\\
&= \frac{\left(ar^2+\frac{m^2-r^2}{n}\right)^{n/2}}{2^n (a+h)^{\frac{n-1}{2}} |mr^{n-1}|}  \\
&= \frac{((na-1)r^2+m^2)^{n/2}}{2^n n^{n/2} (a+h)^{\frac{n-1}{2}}  |mr^{n-1}|} 
\end{align*}
as desired. \end{proof}

The following lemma will turn out useful when we maximize and minimize the expression for the center density derived in the previous proposition.

\begin{lemma}
  \label{lem:calc}
  Let $n\geq2$ be an integer and let $a>\frac{1}{n}$. Define
  \begin{align*}
    l&:=\frac{na-1}{n^2-1}, \\
    u&:=\frac{(na-1)(n+1)}{n-1}.
\end{align*}
  Then the real-valued function $f:[l,u]\rightarrow \R$ defined by
$$f(x)=\frac{(na-1+x)^{n}}{x}$$
has a maximum point at $x=u$ and the maximum value is  $$f(u)=2^{n}n^{n}(na-1)^{n-1}(n-1)^{1-n}(n+1)^{-1}.$$
Further, $f$ has a minimum point at $x=x_0:=\frac{na-1}{n-1}$ and the minimum value is
$$f(x_0)=n^{n}(na-1)^{n-1}(n-1)^{1-n}.$$
\end{lemma}

\begin{proof} As a differentiable function, $f$ achieves its extreme values at the endpoints of $[l,u]$ or at a point where $f'$ vanishes. A simple calculation shows that

\begin{itemize}
    \item $f(l)=n^{2n}(na-1)^{n-1}(n^2-1)^{1-n}$
    \item $f(u)=2^{n}n^{n}(na-1)^{n-1}(n-1)^{1-n}(n+1)^{-1}$
    \item $f'(x)=\frac{(na+x-1)^{n-1} ((n-1)x-na+1)}{x^2}$ and thus $f'(x)=0$ if and only if $$x=1-na \quad \text{or}\quad x=\frac{na-1}{n-1}.$$
\end{itemize}
The first equality is impossible since $na>1$ by assumption, and negative solutions are not allowed since $l>0$. This leaves us with a single zero for the derivative, $x_0=\frac{na-1}{n-1}\in [l,u]$. The value of the function at this point is
$$f(x_0)=n^{n}(na-1)^{n-1}(n-1)^{1-n}.$$
Note that 
$$
    \frac{f(u)}{f(l)} = \frac{2^{n}}{(n+1)^2}\left(1+\frac{1}{n}\right)^{n} 
    \geq \frac{2^{2}}{(2+1)^2}\left(1+\frac{1}{2}\right)^{2}  \
    =1,
$$
when $n\geq2$. Similarly, 
$$
  \frac{f(l)}{f(x_0)}=n^{n}(n+1)^{1-n}   
  \geq 2^2(2+1)^{1-2} 
  =\frac{4}{3}>1,
$$
when $n\geq2$. Since $f$ is positive, it follows that
$$f(x_0)< f(l)\leq f(u),$$
proving the lemma. \end{proof}

We now have the ingredients to find bounds for $\delta(\LLL)$.

\begin{proposition}
 \label{prop:l2}
 Let $n\geq 2$ be an integer and let $\mathcal{L}\subset\R^n$ be a tame lattice with a Lagrangian basis $\set{\B{e}_1,\dots,\B{e}_n}$. Let $a:=\ideal{\B{e}_1,\B{e}_1}$ and $h:=-\ideal{\B{e}_1,\B{e}_2}$. Let $r,s$ be integers such that $0\neq |r|<n$ and let $m=r+sn$. Suppose that 
$$\frac{na-1}{n^2-1}\leq \left(\frac{m}{r}\right)^2\leq\frac{(na-1)(n+1)}{n-1}.$$ Then
 $$\frac{1}{2^n}\leq\delta(\LL_{\B{v}_1}^{(r,s)})\leq \frac{1}{2^{n/2}\sqrt{n+1}}.$$
 Further, the lower bound is achieved when
 $$ \left(\frac{m}{r}\right)^2=\frac{na-1}{n-1}$$
 and the upper bound is achieved when
 $$\left(\frac{m}{r}\right)^2=\frac{(na-1)(n+1)}{n-1}.$$
\end{proposition}

\begin{proof} \ Proposition \ref{prop:l1} gives the center density of $\LLL$ as
\begin{align*}
\delta(\LL_{\B{v}_1}^{(r,s)})&=\frac{((na-1)r^2+m^2)^{n/2}}{2^n n^{n/2} (a+h)^{\frac{n-1}{2}}  |mr^{n-1}|}\\
&=\frac{\left(na-1+\left(\frac{m}{r}\right)^2\right)^{n/2}}{2^n n^{n/2} (a+h)^{\frac{n-1}{2}}  \left|\frac{m}{r}\right|}.
\end{align*}

The condition for $q:=(m/r)^2$ is satisfied when $q \in [l,u]$, where $l$ and $u$ are defined as in Lemma \ref{lem:calc}.  Note that we have
\begin{align*}
    \delta(\LL_{\B{v}_1}^{(r,s)})&=\frac{\sqrt{f\left(q\right)}}{2^n n^{n/2} (a+h)^{\frac{n-1}{2}}},
\end{align*}
where $f:[l,u]\rightarrow\R$ is the function defined in Lemma \ref{lem:calc}. It follows from the previous lemma, where $x_0=\frac{na-1}{n-1}$, that
\begin{align*}
    \delta(\LL_{\B{v}_1}^{(r,s)})&\geq \frac{\sqrt{f(x_0)}}{2^n n^{n/2} (a+h)^{\frac{n-1}{2}}}\\
    &= \frac{\sqrt{f(x_0)}}{2^n n^{n/2} (a+\frac{a-1}{n-1})^{\frac{n-1}{2}} } \\
    &=\frac{n^{n/2}(na-1)^{\frac{n-1}{2}}(n-1)^{\frac{1-n}{2}}}{2^n n^{n/2} (a+\frac{a-1}{n-1})^{\frac{n-1}{2}}}\\
    &=\frac{n^{n/2}(na-1)^{\frac{n-1}{2}}(n-1)^{\frac{1-n}{2}}}{2^n n^{n/2} (na-1)^{\frac{n-1}{2}}(n-1)^{\frac{1-n}{2}}} \\
    &=\frac{1}{2^n},
\end{align*}
and
\begin{align*}
\delta(\LL_{\B{v}_1}^{(r,s)})&\leq\frac{\sqrt{f(u)}}{2^n n^{n/2} (a+h)^{\frac{n-1}{2}} } \\
&= \frac{\sqrt{f(u)}}{2^n n^{n/2} (a+\frac{a-1}{n-1})^{\frac{n-1}{2}} } \\
&=\frac{2^{n/2}n^{n/2}(na-1)^{\frac{n-1}{2}}(n-1)^{\frac{1-n}{2}}(n+1)^{-1/2}}{2^n n^{n/2} (na-1)^{\frac{n-1}{2}}(n-1)^{\frac{1-n}{2}}} \\
&= \frac{1}{2^{n/2}\sqrt{n+1}},
\end{align*}
proving the claim. \end{proof}

\begin{rem}
 \label{rem:2}
 Since $\delta(A_n)=\frac{1}{2^{n/2}\sqrt{n+1}}$ and $\delta(\Z^n)=2^{-n}$, the previous proposition says that $$\delta(\Z^n)\leq\delta(\LL_{\B{v}_1}^{(r,s)})\leq\delta(A_n).$$
 \end{rem}

 Note that under the assumptions made in Theorem \ref{thm:l1}, the family of lattices $\LL_{\B{v}_1}^{(r,s)}$ do not have asymptotically good packing density. In Section \ref{sec:denser}, we will see that under more general assumptions denser packings can be obtained from simple tame lattices. 
 To complete our analysis, we give a characterization of $\LLL$ when the bounds of the previous proposition are achieved.

\begin{theorem}
\label{thm:4}
 Let $n\geq 2$ be an integer and let $\mathcal{L}\subset\R^n$ be a tame lattice with a Lagrangian basis $\set{\B{e}_1,\dots,\B{e}_n}$. Let $a:=\ideal{\B{e}_1,\B{e}_1}$ and $h:=-\ideal{\B{e}_1,\B{e}_2}$. Let $r,s$ be integers such that $0\neq |r|<n$ and let $m=r+sn$. If
\begin{enumerate}[label={(\alph*)}]
    \item\label{thm4a} $$ \left(\frac{m}{r}\right)^2=\frac{na-1}{n-1},$$
then $\LL_{\B{v}_1}^{(r,s)}\cong |r|\sqrt{\frac{na-1}{n-1}}\Z^n$.
   \item\label{thm4b} $$ \left(\frac{m}{r}\right)^2=\frac{(na-1)(n+1)}{n-1},$$
then $\LL_{\B{v}_1}^{(r,s)}\cong |r|\sqrt{\frac{na-1}{n-1}}A_n$.
\end{enumerate}
\end{theorem}

\begin{proof} \ Theorem \ref{thm:l1} gives a basis of minimal vectors $$\mathcal{B}=\set{\Phi_{(r,s)}(\B{e}_i):1\leq i\leq n}=\set{r\B{e}_1+s\B{v}_1,\dots, r\B{e}_n+s\B{v}_1}$$ for the lattice $\LL_{\B{v}_1}^{(r,s)}$. A direct computation shows that
\begin{enumerate}
    \item $\ideal{\Phi_{(r,s)}(\B{e}_i),\Phi_{(r,s)}(\B{e}_i)}=ar^2+\frac{m^2-r^2}{n}$ for all $1\leq i\leq n$.
    \item $\ideal{\Phi_{(r,s)}(\B{e}_i),\Phi_{(r,s)}(\B{e}_j)}=-r^2h+\frac{m^2-r^2}{n}$ for all $1\leq i\neq j\leq n$.
\end{enumerate}

For \ref{thm4a}, using the equations $a+h(1-n)=1$ and $m^2=r^2\frac{na-1}{n-1}$ we end up with
$$\ideal{\Phi_{(r,s)}(\B{e}_i),\Phi_{(r,s)}(\B{e}_j)}=\begin{cases}\frac{(na-1)r^2}{n-1}\quad \text{if } i=j \\
0\qquad\quad \ \ \ \text{if } i\neq j.
\end{cases}$$
As a consequence, the Gram matrix of the lattice $\LL_{\B{v}_1}^{(r,s)}$ with respect to the basis $\mathcal{B}$ is given by
$$G_{\mathcal{B}}= \frac{(na-1)r^2}{n-1}I_n.$$
Therefore, $\LL_{\B{v}_1}^{(r,s)}\cong |r|\sqrt{\frac{na-1}{n-1}}\Z^n$ proving \ref{thm4a}.
For \ref{thm4b}, using the equations $a+h(1-n)=1$ and $m^2=r^2\frac{(na-1)(n+1)}{n-1}$ we end up with
$$\ideal{\Phi_{(r,s)}(\B{e}_i),\Phi_{(r,s)}(\B{e}_j)}=\begin{cases}\frac{2(na-1)r^2}{n-1}\quad \text{if } i=j \\
\frac{(na-1)r^2}{n-1}\quad \ \text{if } i\neq j.
\end{cases}$$
As a consequence, the Gram matrix for the lattice $\LL_{\B{v}_1}^{(r,s)}$ with respect to the basis $\mathcal{B}$ is given by
$$G_{\mathcal{B}}= \frac{(na-1)r^2}{n-1}\left[
  \begin{array}{cccc}
    2 & 1 & \ldots & 1 \\
   1 & 2 & \ddots & \vdots \\
    \vdots & \ddots & \ddots & 1 \\
    1 & \ldots & 1 & 2 \\
  \end{array}
\right]=:\frac{(na-1)r^2}{n-1}A.$$
Define the unimodular matrix 
$$U:=\left[
  \begin{array}{rrrr}
    1 &0 & \ldots & 0 \\
   -1 & 1 & \ddots & \vdots \\
    \vdots & \ddots & \ddots & 0 \\
    0 & \ldots & -1 & 1 \\
  \end{array}
\right].$$

Note that
$$\frac{n-1}{(na-1)r^2}UG_{\mathcal{B}}U^T=UAU^T=G$$
where
$$G=\left[
  \begin{array}{rrrr}
    2 & -1 & \ldots & 0 \\
   -1 & 2 & \ddots & \vdots \\
    \vdots & \ddots & \ddots & -1 \\
    0 & \ldots & -1 & 2 \\
  \end{array}
\right]$$
is a Gram matrix for $A_n$. Therefore, $\LL_{\B{v}_1}^{(r,s)}\cong |r|\sqrt{\frac{na-1}{n-1}}A_n$ proving \ref{thm4b}. \end{proof}

\subsection{Construction of $A_n$}

We know that a lattice of the form $\LLL$ is similar to $A_n$ if condition \ref{thm4b} in Theorem \ref{thm:4} holds. But the question remains, for which values of the parameter $a$ for the tame superlattice does there exist a pair of integers $(r,s)$ such that the condition is satisfied. If we take the orthogonal lattice, with $a=1$, the condition can only be satisfied when $n+1$ is a square. But if we consider a tame lattice with $a=n$, then we can always find pairs $(r,s)$ such that $\LLL$ is similar to $A_n$. Propositions \ref{prop:3a} and \ref{prop:3} show how one can construct a lattice similar to the $A_n$ lattice as a sublattice of $\Z^n$ in the case that $n+1$ is a square, and as a sublattice of a specific tame lattice for a general $n$.

\begin{proposition}
  \label{prop:3a}
   Suppose that $n+1=d^2$ for some integer $d>2$. Let $\LL=\Z^n$ be the tame lattice with $(a,h)=(1,0)$ and a Lagrangian basis $\set{\B{e}_1,\dots,\B{e}_n}$, where the $\B{e}_i$ are the standard basis vectors in $\R^n$. Then $\LL_{\B{v}_1}^{(d+1,1)}$ is a sublattice of $\LL$ such that $\LL_{\B{v}_1}^{(d+1,1)}\cong (d+1)A_n$.
\end{proposition}

\begin{proof} Let $(r,s)=(d+1,1)$, and note that $0<|r|=d+1<(d+1)(d-1)=n$. Moreover,
$$\left(\frac{m}{r}\right)^2=\left(\frac{r+sn}{r}\right)^2=\left(1+\frac{n}{r}\right)^2=d^2=n+1=\frac{(na-1)(n+1)}{n-1}.$$
By Theorem \ref{thm:4}, part \ref{thm4b},
$$\LL_{\B{v}_1}^{(r,s)}\cong (d+1) A_n.$$
\end{proof}

\begin{proposition}
  \label{prop:3}
  Let $n\geq 2$ be an integer and let $\B{e}_1,\dots,\B{e}_n$ be the standard basis vectors in $\R^n$. Let $\B{v}_1=(1,\dots, 1)^T\in \R^n$ and define the vectors
  $\B{e}_i'=\left(\frac{1-\sqrt{n+1}}{n}\right)\B{v}_1+\sqrt{n+1}\B{e}_i$ for each $i=1,\dots, n$. Then the lattice $\LL\subset \R^n$ with the basis $\set{\B{e}_1',\dots,\B{e}_n'}$ is tame with $(a,h)=(n,1)$. Further, for any integer $r$ such that $0<|r|<n$, we have that
  $\LL_{\B{v}_1}^{(r,r)}$ is a sublattice of $\LL$ such that $\LL_{\B{v}_1}^{(r,r)}\cong|r|\sqrt{n+1}A_n$.
\end{proposition}

\begin{proof} \ First we need to check that the vectors $\B{e}_i'$ form a Lagrangian basis.
\begin{enumerate}
    \item $\sum_{i=1}^n \B{e}_i'= n\left(\frac{1-\sqrt{n+1}}{n}\right)\B{v}_1+\sqrt{n+1}\B{v}_1 = \B{v}_1$.
    \item $\ideal{\B{e}_i', \B{v}_1}=n\left(\frac{1-\sqrt{n+1}}{n}\right)+\sqrt{n+1}=1$ for all $1\leq i\leq n$.
    \item $a=\ideal{\B{e}_i',\B{e}_i'}=(n-1)\left(\frac{1-\sqrt{n+1}}{n}\right)^2+\left(\frac{1-\sqrt{n+1}}{n}+\sqrt{n+1}\right)^2=n$ for all $1\leq i\leq n$.
    \item $-h=\ideal{\B{e}_i',\B{e}_j'}=(n-2) \left(\frac{1-\sqrt{n+1}}{n}\right)^2 + 2\left(\frac{1-\sqrt{n+1}}{n}\right)\left(\frac{1-\sqrt{n+1}}{n}+\sqrt{n+1}\right)=-1$ for all $1\leq i\neq j \leq n$.
\end{enumerate}

The conditions are satisfied, so $\LL$ is tame with $(a,h)=(n,1)$. For the second part, suppose that $r$ is an integer such that $0<|r|<n$. Let $s:=r$, so that $m=r(1+n)$, and note that
$$\left(\frac{m}{r}\right)^2=(1+n)^2=\frac{(n^2-1)(n+1)}{n-1}=\frac{(na-1)(n+1)}{n-1}.$$
By Theorem \ref{thm:4}, part \ref{thm4b},
$$\LL_{\B{v}_1}^{(r,r)}\cong |r|\sqrt{\frac{na-1}{n-1}}A_n=|r|\sqrt{n+1}A_n$$
and we are done. \end{proof}

\subsection{Dual lattices}

In this section we show that the dual of a lattice of the form $\LLL$ is a scaled version of $\LL$ (Proposition \ref{prop:5}), and that the dual of a tame lattice is also tame (Lemma \ref{lem:lag_dual}). The most general result is Theorem \ref{DualOfLinearLattices}, which says that lattices of the form $\LLLL$, with some conditions on the parameters $r,s,$ are closed, modulo scaling, under duality.  

\begin{proposition}
  \label{prop:5}
  Let $n\geq 2$ be an integer and let $\mathcal{L}\subset\R^n$ be a tame lattice with a Lagrangian basis $\set{\B{e}_1,\dots,\B{e}_n}$. Let $a:=\ideal{\B{e}_1,\B{e}_1}$ and $h:=-\ideal{\B{e}_1,\B{e}_2}$. Let $r,s$ be integers such that $0\neq |r|<n$. Suppose that $s=rh$. Then the dual lattice of $\LL_{\B{v}_1}^{(r,s)}$ is $\frac{1}{r(a+h)}\LL$.
\end{proposition}

\begin{proof} \ A basis for $\LL_{\B{v}_1}^{(r,s)}$ is given by the vectors $r\B{e}_i+s\B{v}_1$, for $i=1,\dots,n$. Let $c:=\frac{1}{r(a+h)}$. Then a basis for $\frac{1}{r(a+h)}\LL$ is given by the vectors $c\B{e}_1,\dots,c\B{e}_n$. Note that for every $i,j\in\set{1,\dots,n}$,
$$\ideal{r\B{e}_i+s\B{v}_1,c\B{e}_j}=\begin{cases}
c(ra+s),\quad\text{ if } i=j \\
c(s-rh),\quad\text{ if }i\neq j,
\end{cases}$$
where we used the fact that $\ideal{\B{v}_1,\B{e}_j}=1$ for all $1\leq j\leq n$, and that $\ideal{\B{e}_i,\B{e}_j}=a$ if $i=j$ and $\ideal{\B{e}_i,\B{e}_j}=-h$ else. The assumption $s=rh$ and $c\neq0$ give that
$$\ideal{r\B{e}_i+s\B{v}_1,c\B{e}_j}=c(ra+s)\delta_{ij}=cr(a+h)\delta_{ij}=\delta_{ij}$$
for all $1\leq i\neq j\leq n$, proving that the generator matrices $M$ and $M'$ of $\LL_{\B{v}_1}^{(r,s)}$ and $\frac{1}{r(a+h)}\LL$, respectively, are related by $M^TM'=I_n$. Therefore, $\LL_{\B{v}_1}^{(r,s)}$ is the dual lattice of $\frac{1}{r(a+h)}\LL$ by Remark \ref{rem:dual}.
\end{proof}

\begin{lemma}
	\label{lem:lag_dual}
	Let $n\geq2$ be an integer and let $\LL\subset \R^n$ be a tame lattice with a Lagrangian basis $\set{\B{e}_1,\dots,\B{e}_n}$. Let $a:=\ideal{\B{e}_1,\B{e}_1}$ and $h:=-\ideal{\B{e}_1,\B{e}_2}$. Then the dual lattice $\LL^*$ is a tame lattice with a Lagrangian basis $\set{\B{e}_1',\dots,\B{e}_n'}$, such that $\tilde{a}:=\ideal{\B{e}_1',\B{e}_1'}=\frac{1+h}{a+h}$ and $\tilde{h}:=-\ideal{\B{e}_1',\B{e}_2'}=\frac{-h}{a+h}$.
\end{lemma}

\begin{proof} First note that $$\tilde{a}-(n-1)\tilde{h}=\frac{1+h}{a+h}-(n-1)\left(\frac{-h}{a+h}\right)=\frac{1+hn}{a+h}=\frac{a+h}{a+h}=1,$$
since $a-h(n-1)=1$, \emph{i.e.}, $1+hn=a+h$, by the definition of a tame lattice. Furthermore, $$\tilde{a}+\tilde{h}=\frac{1+h}{a+h}+\frac{-h}{a+h}=\frac{1}{a+h}>0$$
since $a+h>0$ by Remark \ref{rem:1}. If we can show that $\LL^*$ has a Gram matrix $\tilde{G}$ of the form
$$\tilde G= \left[
\begin{array}{cccc}
	\tilde{a} & -\tilde{h} & \ldots & -\tilde{h} \\
	-\tilde{h} & \tilde{a} & \ddots & \vdots \\
	\vdots & \ddots & \ddots & -\tilde{h} \\
	-\tilde{h} & \ldots & -\tilde{h} & \tilde{a} \\
\end{array}
\right],$$
 with respect to some basis $\set{\B{e}_1',\dots,\B{e}_n'}$, then we have shown that $\LL^*$ is a tame lattice with a Lagrangian basis $\set{\B{e}_1',\dots,\B{e}_n'}$ such that $\tilde{a}=\ideal{\B{e}_1',\B{e}_1'}=\frac{1+h}{a+h}$ and $\tilde{h}=-\ideal{\B{e}_1',\B{e}_2'}=\frac{-h}{a+h}$. Denote by $\set{\B{b}_1,\dots,\B{b}_n}$ the rows of the Gram matrix $G$ of the tame lattice $\LL$, and by $\set{\B{d}_1,\dots,\B{d}_n}$ the columns of the matrix $\tilde{G}$. Let us compute the inner products
$\ideal{\B{b}_i,\B{d}_j}$ for all $i,j\in\set{1,\dots,n}$. For $i=j$, we obtain
\begin{align*}
	\ideal{\B{b}_i,\B{d}_i} &= (-h,\dots,a,\dots,-h)^T(-\tilde h,\dots,\tilde a,\dots, -\tilde h)\\
	&= (n-1)h\tilde{h}+a\tilde{a} \\
	&= (n-1)h\tilde{h}+a(1+(n-1)\tilde{h}) \\
	&= a+(n-1)\tilde{h}(a+h) \\
	&= a+(n-1)\left(\frac{-h}{a+h}\right)(a+h) \\
	&= a-h(n-1)=1.
\end{align*}
For $i\neq j$, we have
\begin{align*}
	\ideal{\B{b}_i,\B{d}_j} &= (-h,\dots,a,-h\dots,-h)^T(-\tilde h,\dots,-\tilde h,\tilde a,\dots, -\tilde h) \\
	&= (n-2)h\tilde{h}-\tilde{a}h-a\tilde{h} \\
	&= \tilde{h}((n-2)h-a)-\tilde{a}h \\
	&= -\tilde{h}(a-h(n-1)+h)-\tilde{a}h \\
	&= -\tilde{h}(1+h)-\tilde{a}h \\
	&= \left(\frac{h}{a+h}\right)(1+h)-\frac{h(1+h)}{a+h}\\
	&= 0.
\end{align*}
This shows that $G$ and $\tilde{G}$ are related by $G\tilde{G}=I_n$, so $\tilde{G}$ must be a Gram matrix for the dual lattice $\LL^*$, by Remark \ref{rem:dual}. \end{proof}

 More generally in the next theorem we give conditions under which the family of lattices of the form $\LLLL$ is closed, modulo scaling,  under duality. Before doing so we recall the notion of dual map. Given two lattices $\Lambda_{1}$ and $\Lambda_{2}$, and a linear map between them $\phi: \Lambda_{1} \to \Lambda_{2}$, there exists a unique linear map $\phi^{*}: \Lambda_{2}^{*} \to \Lambda_{1}^{*}$ that is completely characterized by the following adjoint property: for every $l_{1} \in \Lambda_{1}$ and $l_{2}^* \in \Lambda_{2}^*$ 
 we have that  \[\langle\phi^{*}(l_{2}^{*}), l_{1} \rangle_{1}= \langle l_{2}^{*}, \phi(l_{1}) \rangle_{2}\] where $\langle \ , \ \rangle_{i}$ denotes the inner product in the vector space $\Lambda_{i} \otimes \R$. If we pick  bases for $\Lambda_{i}$, together with their respective dual basis,  we can identify ${\rm Hom}(\Lambda_{1}, \Lambda_{2})$ with ${\rm M}_{m,n}(\Z)$. Under this choice of coordinates the map 
\begin{align*}
{\rm Hom}(\Lambda_{1}, \Lambda_{2}) & \to  {\rm Hom}(\Lambda_{2}^{*}, \Lambda_{1}^{*}) \\ 
  \phi & \mapsto  \phi^{*}
\end{align*}
is just transposition,
\begin{align*}
 {\rm M}_{m,n}(\Z) & \to  {\rm M}_{n,m}(\Z) \\ 
  A & \mapsto  A^{T}.
\end{align*}

Observe that in particular we have that the map $\phi \mapsto \phi^{*}$ is linear and that $(\phi \circ \psi)^{*}= \psi^{*} \circ \phi^{*}$.

\begin{theorem}\label{DualOfLinearLattices}
  Let $\LL\subset\R^n$ be a tame lattice, let $\T:\LL\rightarrow\Z$ be a non-trivial linear map such that $\B{v}_1\in\LL\setminus\ker\T$. Let $r,s$ be integers such that $r\neq 0$ and $m:= r+s\T(\B{v}_1)\neq 0$. Suppose that $r+m=0$. Then, the dual lattice of $\LLLL$ is, up to scaling, of the same type.  More explicitly, there exists a rank $n$ tame lattice $\LL_{1}$, a non-trivial linear map  $\widetilde{\T}:\LL_{1}\rightarrow\Z$ and  $\widetilde{\B{v}}_1\in\LL_{1}\setminus\ker\widetilde{\T}$, with $r+s\widetilde{\T}(\B{v}_1)\neq 0$,  
  such that $\left(\LLLL\right)^{*}$ is   $\displaystyle\frac{1}{r^2}(\mathcal{L}_{1})_{\widetilde{\T},\widetilde{\mathbf{v}}_1}^{(r,s)}$. 
\end{theorem}

\begin{proof}
By definition, see Remark \ref{Condition r m}, $\LLLL$ is the rank $n$ lattice defined as the image of the map $\Phi_{(r,s)}: \LL \to \LL$;  $\B{x} \mapsto r\phi_{1}(\B{x})+s\phi_{2}(\B{x})$, where $\phi_{1}(\B{x})=\B{x}$ and $\phi_{2}(\B{x})=\rm{T}(\B{x})\B{v}_1$. Since dualizing  is a linear operator we see that $\Phi_{(r,s)}^{*}=r\phi_{1}^{*}+s\phi_{2}^{*}$. Notice that the map $\phi_{2}: \LL \to \LL; \ \phi_{2}(\B{x})=\rm{T}(\B{x})\B{v}_1$ can be written as the composition of the maps  $ {\rm T}: \LL \to \Z $ and $ \mu_{\B{v}_{1}}: \Z \to \LL; \  n \mapsto n \B{v}_{1}$ \[\phi_{2}=\mu_{\B{v}_{1}}\circ \T.\] 

Hence, \[\phi_{2}^{*}=  {\rm T}^{*} \circ \mu_{\B{v}_{1}}^{*} .\] Let $\widetilde{\T}=\mu_{\B{v}_{1}}^{*}$ and $\widetilde{\B{v}}_1:= {\rm T}^{*}(1)$ (here we are using that $\Z^{*}=\Z$). The image of ${\rm T}^{*} : \Z^{*} \to \LL^{*}$ is generated by $\widetilde{\B{v}}_1$ so we can write, in a similar fashion as above, $\mu_{\widetilde{\B{v}}_1}:={\rm T}^{*}$ and obtain that \[\phi_{2}^{*}=  \mu_{\widetilde{\B{v}}_{1}}\circ \widetilde{\T}.\]

It follows from this that $\Phi_{(r,s)}^{*}(\B{x})=r\phi_{1}^{*}(\B{x})+s\phi_{2}^{*}(\B{x})=r\B{x}+s\widetilde{\T}(\B{x}) \widetilde{\B{v}}_1.$ Thanks to Lemma \ref{lem:lag_dual} we have that $\LL_{1}:=\LL^{*}$ is a tame lattice, and since $\widetilde{\T}(\widetilde{\B{v}}_1)=\mu_{\B{v}_{1}}^{*} \circ {\rm T}^{*}(1)= ({\rm T} \circ \mu_{\B{v}_{1}})^{*}(1)={\rm T}(\B{v}_{1})$, we see that $r+s\widetilde{\T}(\widetilde{\B{v}}_1)\neq 0$. In particular, $(\mathcal{L}_{1})_{\widetilde{\T},\widetilde{\mathbf{v}}_1}^{(r,s)}$ is well-defined and it has, by definition, a $\Z$-basis of the form $\{\Phi_{(r,s)}^{*}(\B{e}^{*}_{1}),\dots,\Phi_{(r,s)}^{*}(\B{e}^{*}_{n})\}$ where $\{\B{e}_{1},\dots,\B{e}_{n}\}$ is a $\Z$-basis of $\LL$. To finish the proof of the theorem we must show that  $\left \langle \Phi_{(r,s)}(\B{e}_{i}), \Phi_{(r,s)}^{*}(\B{e}^{*}_{j})\right \rangle =r^2\delta_{i,j}$ for some basis. Let $\{\B{e}_{1},...,\B{e}_{n}\}$ be a Lagrangian basis of $\LL$ with $\B{v}_1=\B{e}_{1}+\cdots+\B{e}_{n}$ and let $\{\B{e}^{*}_{1},\dots,\B{e}^{*}_{n}\}$ be its dual basis.

\begin{align*}
\left \langle \Phi_{(r,s)}(\B{e}_{i}), \Phi_{(r,s)}^{*}(\B{e}^{*}_{j})\right \rangle &= \left \langle r\B{e}_{i}+s\B{v}_{1}\T(\B{e}_{i}), \Phi_{(r,s)}^{*}(\B{e}^{*}_{j})\right \rangle     \\
&=
r\left \langle \B{e}_{i}, \Phi_{(r,s)}^{*}(\B{e}^{*}_{j})\right \rangle + s\T(\B{e}_{i})\left \langle \B{v}_{1}, \Phi_{(r,s)}^{*}(\B{e}^{*}_{j})\right \rangle\\ &= r\left \langle \Phi_{(r,s)}(\B{e}_{i}), \B{e}^{*}_{j}\right \rangle + s\T(\B{e}_{i})\left \langle \Phi_{(r,s)}(\B{v}_{1}), \B{e}^{*}_{j}\right \rangle   \\ &
 =r\left \langle \Phi_{(r,s)}(\B{e}_{i}), \B{e}^{*}_{j}\right \rangle + s\T(\B{e}_{i}) \sum_{k=1}^{n}\left\langle \Phi_{(r,s)}(\B{e}_{k}), \B{e}^{*}_{j}\right \rangle \\ &= r\left \langle r\B{e}_{i}+s\B{v}_{1}\T(\B{e}_{i}), \B{e}^{*}_{j}\right \rangle + s\T(\B{e}_{i}) \sum_{k=1}^{n}\left\langle \Phi_{(r,s)}(\B{e}_{k}), \B{e}^{*}_{j}\right \rangle \\ & =
r\left \langle r\B{e}_{i}+s\B{v}_{1}\T(\B{e}_{i}), \B{e}^{*}_{j}\right \rangle + s\T(\B{e}_{i}) \sum_{k=1}^{n}\left\langle r\B{e}_{k}+s\B{v}_{1}\T(\B{e}_{k}), \B{e}^{*}_{j}\right \rangle \\ &= r^2\delta_{i,j}+rs\T(\B{e}_{i})+ s\T(\B{e}_{i})\left(r+s \sum_{k=1}^{n}\T(\B{e}_{k}) \right)  \\ &=
 r^2\delta_{i,j}+rs\T(\B{e}_{i})+ s\T(\B{e}_{i})\left(r+s \T(\B{v}_{1}) \right) \\ &= r^2\delta_{i,j}+rs\T(\B{e}_{i})+ s\T(\B{e}_{i})m   =
 r^2\delta_{i,j}+s\T(\B{e}_{i})(r+ m) = r^2\delta_{i,j}.
\end{align*}

\end{proof}

\subsection{Construction of some lattices with high center density}
\label{sec:denser}
We have seen that under the assumptions made in Theorem \ref{thm:l1}, a well-rounded lattice of the form $\LLL$ has center density less than or equal to the center density of $A_n$. However, if we replace $\T_{\B{v}_1}(\B{x})=\ideal{\B{x},\B{v}_1}$ with a general $\Z$-linear map $\T:\LL\rightarrow\Z$, then we get lattices with higher center density. Indeed, in \cite[Example 3.6]{damir2020bases}, using the trace map $\T:\LL\rightarrow\Z$, $\T(\B{x})=\sum_{i=1}^n x_i$, the authors construct the lattice $D_n$ as a lattice of the form $\LLLL$. The purpose of this section is to show how one can construct the densest known lattices in dimensions 8 and 9 as lattices $\LLLL$, when $\LL$ is the tame lattice $\Z^n$. This shows that there are many different types of lattices that can be obtained from a tame lattice $\LL$, using the linear map $\Phi_{(r,s)}$. 

\subsubsection{The $E_8$ lattice}

The $E_8$ lattice is the densest lattice packing in dimension 8. There are two equivalent versions of the lattice; the even ($\Gamma_8$) and odd ($\Gamma_8'$) coordinate system version. Let us state the definition of both.

\begin{definition}
\label{def:gamma}
Define 
\begin{align}
\Gamma_8&:=\set{\B{x}\in\Z^8\cup\left(\Z+\frac{1}{2}\right)^8:\sum_{i=1}^8 x_i\in2\Z}, \\
 \Gamma_8'&:=\set{\B{x}\in\Z^8:\sum_{i=1}^8x_i\in2\Z}\cup\set{\B{x}\in\left(\Z+\frac{1}{2}\right)^8:\sum_{i=1}^8x_i\in2\Z+1}.
\end{align}
\end{definition}

Proposition \ref{prop:rigid_basis} gives the index of $\LLLL$ in $\LL$; this result will become useful in Lemma \ref{lem:e8}. 

\begin{proposition}
  \label{prop:rigid_basis}
  Let $\LL\subset\R^n$ be a lattice, $\T:\LL\rightarrow\Z$ a non-trivial linear map and $\B{v}_1\in\LL\setminus\ker\T$. Let $r,s$ be integers such that $0\neq |r|<|\T(\B{v}_1)|$ and let $m=r+s\T(\B{v}_1)$. Then $$[\LL:\LLLL]=|mr^{n-1}|.$$
\end{proposition}


\begin{proof} The case when $\LL$ has a basis which is rigid with respect to $\T$ is proven in \cite[Proposition 3.7]{damir2020bases}, so let us prove that such a basis always exists. Since $\T:\LL\rightarrow \Z$ is non-trivial there is $\B{v} \in \LL$ such that ${\rm Im}(\T)$ is generated by $\T(\B{v})$, moreover by the additivity of the rank $\ker \T$ is a sublattice of $\LL$ of rank $n-1$. Let $\set{\B{w}_1,\dots,\B{w}_{n-1}}$ be a basis for $\ker \T$. Notice that $\set{\B{w}_1,\dots,\B{w}_{n-1}, \B{v}}$ is a basis for $\LL$. Thus,  $\set{\B{w}_1+\B{v},\dots,\B{w}_{n-1}+\B{v}, \B{v}}$  is a basis for $\LL$ and it is rigid with respect to $\T$ since $\T$ takes the value of $\T(\B{v})$ at every element in the set.


\end{proof}

\begin{lemma}
\label{lem:e8}
  Let $\LL=\Z^8$ and $\B{c}=(1,-1,1,-1,1,-1,1,-1)$ in the standard basis. Define the linear map $\T:\LL\rightarrow \Z$,
  $\T(\B{x})=\langle \B{c}, \B{x}\rangle$. Let $\B{v}_1=(-1,1,-1,1,1,1,1,1)\in \LL\setminus{\ker \T}$ in the standard basis, and $(r,s)=(2,1)$. Then $\LL_{\T,\B{v}_1}^{(r,s)}=2\Gamma_8'$.
\end{lemma}

\begin{proof} First note that $0\neq |r|=2<4=|\T(\B{v}_1)|$, so Lemma \ref{lem:l1} gives that $\LL_{\T,\B{v}_1}^{(r,s)}$ is a full rank sublattice of $\LL$. Let $m:=r+s\T(\B{v}_1)=2+(-4)=-2$. Let $\B{e}_1,\dots,\B{e}_8$ be the standard basis vectors in $\R^8$. A basis for $\LL_{\T,\B{v}_1}^{(r,s)}$ is given by the vectors
$$\Phi_{(r,s)}(\B{e}_i)=r\B{e}_i+s\T(\B{e}_i)\B{v}_1=2\B{e}_i+c_i\B{v}_1$$
producing the generator matrix
$$M=\left[
\begin{array}{rrrrrrrrr}
 1 & 1 & -1& 1 & -1 & 1 & -1 & 1  \\
 1 & 1 & 1 & -1 & 1 & -1 & 1 & -1 \\
 -1 & 1 & 1 & 1 & -1 & 1 & -1 & 1  \\
 1 & -1 & 1 & 1 & 1 & -1 & 1 & -1  \\
 1 & -1 & 1 & -1 & 3 & -1 & 1 & -1  \\
 1 & -1 & 1 & -1 & 1 & 1 & 1 & -1  \\
 1 & -1 & 1 & -1 & 1 & -1 & 3 & -1 \\
 1 & -1 & 1 & -1 & 1 & -1 & 1 & 1  \\
\end{array}
\right].$$
If we add column $i$ to column $i+1$ for every $i\in\set{1,\dots,7}$, we get the following generator matrix:
$$M'=\left[
\begin{array}{rrrrrrrr}
 1 & 2 & 0 & 0 & 0 & 0 & 0 & 0 \\
 1 & 2 & 2 & 0 & 0 & 0 & 0 & 0  \\
 -1 & 0 & 2 & 2 & 0 & 0 & 0 & 0 \\
 1 & 0 & 0 & 2 & 2 & 0 & 0 & 0  \\
 1 & 0 & 0 & 0 & 2 & 2 & 0 & 0  \\
 1 & 0 & 0 & 0 & 0 & 2 & 2 & 0  \\
 1 & 0 & 0 & 0 & 0 & 0 & 2 & 2  \\
 1 & 0 & 0 & 0 & 0 & 0 & 0 & 2  \\
\end{array}
\right].$$

Since the columns of $M'$ are contained in $2\Gamma_8'$, we can conclude that $\LL_{\T,\B{v}_1}^{(r,s)}\subseteq 2\Gamma_8'$. On the other hand, since $\det(\LL)=1$ and $[\LL:\LL_{\T,\B{v}_1}^{(r,s)}]=|mr^{n-1}|$ by Proposition \ref{prop:rigid_basis}, we have $$\det(\LL_{\T,\B{v}_1}^{(r,s)})=[\LL:\LL_{\T,\B{v}_1}^{(r,s)}]^2\det(\LL)=|mr^{n-1}|^2=|(-2)\cdot 2^{8-1}|^2=4^8=\det(2\Gamma_8'),$$
proving $\LL_{\T,\B{v}_1}^{(r,s)}=2\Gamma_8'$.
\end{proof}

\subsubsection{Densest known lattice packing in dimension 9}

The largest known center density of a lattice packing in dimension 9 is $\delta_9=\frac{1}{16\sqrt2}$, achieved by the laminated lattice $\Lambda_9$ \cite{nebe_sloane_2012}. Here we show how we can construct a lattice $\LLLL$ with this center density.

\begin{lemma}
  Let $\LL=\Z^9$ and $\B{c}=(1,-1,1,-1,1,-1,1,-1,1)$ in the standard basis. Define the linear map $\T:\LL\rightarrow \Z$, $\T(\B{x})=\B{c}^T \B{x}$. Let $$\B{v}_1=(-1, -1, -1, -1, -1, -1, -2, 1, -1)\in \LL\setminus{\ker \T}$$ in the standard basis, and $(r,s)=(2,1)$. Then $\delta(\LL_{\T,\B{v}_1}^{(r,s)})=\frac{1}{16\sqrt2}=\delta_{9}$.
\end{lemma}

\begin{proof} First note that $0\neq |r|=2<4=|\T(\B{v}_1)|$, so $\LL_{\T,\B{v}_1}^{(r,s)}$ is a full rank sublattice of $\LL$. Let $\B{e}_1,\dots,\B{e}_9$ be the standard basis vectors in $\R^9$. A basis for $\LL_{\T,\B{v}_1}^{(r,s)}$ is given by the vectors
$$\Phi_{(r,s)}(\B{e}_i)=r\B{e}_i+s\T(\B{e}_i)\B{v}_1=2\B{e}_i+c_i\B{v}_1.$$
Thus, a generator matrix for $\LL_{\T,\B{v}_1}^{(r,s)}$ is given by
$$M=\left[
\begin{array}{rrrrrrrrr}
 1 & 1 & -1 & 1 & -1 & 1 & -1 & 1 & -1 \\
 -1 & 3 & -1 & 1 & -1 & 1 & -1 & 1 & -1 \\
 -1 & 1 & 1 & 1 & -1 & 1 & -1 & 1 & -1 \\
 -1 & 1 & -1 & 3 & -1 & 1 & -1 & 1 & -1 \\
 -1 & 1 & -1 & 1 & 1 & 1 & -1 & 1 & -1 \\
 -1 & 1 & -1 & 1 & -1 & 3 & -1 & 1 & -1 \\
 -2 & 2 & -2 & 2 & -2 & 2 & 0 & 2 & -2 \\
 1 & -1 & 1 & -1 & 1 & -1 & 1 & 1 & 1 \\
 -1 & 1 & -1 & 1 & -1 & 1 & -1 & 1 & 1 \\
\end{array}
\right].$$
If we add column $i$ to column $i+1$ for every $i\in\set{1,\dots,8}$, we get the generator matrix
$$M'=\left[
\begin{array}{rrrrrrrrr}
 1 & 2 & 0 & 0 & 0 & 0 & 0 & 0 & 0 \\
 -1 & 2 & 2 & 0 & 0 & 0 & 0 & 0 & 0 \\
 -1 & 0 & 2 & 2 & 0 & 0 & 0 & 0 & 0 \\
 -1 & 0 & 0 & 2 & 2 & 0 & 0 & 0 & 0 \\
 -1 & 0 & 0 & 0 & 2 & 2 & 0 & 0 & 0 \\
 -1 & 0 & 0 & 0 & 0 & 2 & 2 & 0 & 0 \\
 -2 & 0 & 0 & 0 & 0 & 0 & 2 & 2 & 0 \\
 1 & 0 & 0 & 0 & 0 & 0 & 0 & 2 & 2 \\
 -1 & 0 & 0 & 0 & 0 & 0 & 0 & 0 & 2 \\
\end{array}
\right].$$
One can verify that $\lambda_1(\LL_{\T,\B{v}_1}^{(r,s)})=\sqrt{2^2+2^2}=\sqrt{8}$ and $\vol{\LL_{\T,\B{v}_1}^{(r,s)}}=2^9$. Therefore,
$$\delta(\LL_{\T,\B{v}_1}^{(r,s)})=\frac{\lambda_1(\LL_{\T,\B{v}_1}^{(r,s)})^{9}}{2^9\vol{\LL_{\T,\B{v}_1}^{(r,s)}}}=\frac{(\sqrt{8})^9}{2^9\cdot 2^9}=\frac{2^9\cdot (\sqrt2) ^9}{2^9\cdot 2^9}=\frac{1}{16\sqrt2}=\delta_9.$$
\end{proof}

\section{Generic well-rounded lattices}
\label{sec:gwr}
In the previous section we saw how to construct GWR sublattices of tame lattices with center density ranging from $\delta(\mathbb{Z}^n)$ to $\delta(A_n)$, so maximally dense lattice packings in dimension 1--3. We proceed to construct GWR lattices which have a good packing density in higher dimensions. We deform basis vectors of the densest lattice packing in dimension 3--5, $D_n$, and in dimension 8, $E_8$, to obtain lattices with good sphere packing density and minimal kissing number in these dimensions. We call these constructions $D_n^\alpha$ and $E_8^\alpha$, where $\alpha$ is a parameter describing how much the basis vectors are distorted. We also investigate when it is possible to scale these deformed lattices to embed them as sublattices of $\Z^n$.

As mentioned earlier, one motivation for finding GWR lattices with high sphere packing density has to do with the theta series of a lattice (cf. Equation \eqref{eq:theta}). Now if we approximate the series with 
$$\Theta_\Lambda(q:=e^{-\frac{1}{2\sigma^2}})\approx1+k_1q^{l_1^2}=1+\kappa(\Lambda)e^{-\frac{\lambda_1(\Lambda)^2}{2\sigma^2}},$$ we see that the kissing number and shortest vector length of a lattice play important roles on the theta series.

The minimization of the theta series for specific ranges for $q$ is interesting in itself, but this problem has also real-life applications in contexts such as wireless communications (cf. Section \ref{sec:app}). The minimization of the theta series also has applications in physics: \emph{e.g.}, finding optimal arrangements of particles interacting under Gaussian potentials \cite{blanc2015}.

Densest lattice packings have a large kissing number in all dimensions for which the densest packings are known. Thus, a natural direction is to  consider lattices with minimal kissing number but for which $\lambda_1$ is close to the maximum in that dimension. This is the motivation behind the deformed lattices. As we will see, however, this does not mean that the theta function of the deformed lattice will be smaller than those for the non-deformed lattices $D_n$ and $E_8$.
Indeed, it has been recently shown in \cite{cohn2019universal} that in dimensions 8 and 24, minimizers of the theta series in these dimension are the $E_8$ and Leech lattice, respectively. Nevertheless, the construction of GWR lattices with arbitrary good packing density is something that, as far as we are aware, has not been done before. Further, such constructions provide good candidates for secure lattice codes according to \cite{towards} (see also discussion in Section \ref{sec:conclusion}).


\subsection{The planar case}

To highlight the idea behind the deformed lattices, we illustrate how the densest lattice packing in dimension 2, the hexagonal lattice $\Lambda_h\sim A_2$, can be deformed to produce GWR lattices with density as close as desired to the optimal density in dimension 2. In fact, we end up with a parametrization of representatives of equivalence classes of planar well-rounded lattices.

\begin{definition}
	Let $0\leq\alpha\leq \frac{1}{2}$ and $\overline{\alpha}:=\sqrt{1-\alpha^2}$. Define $\Lambda_h^\alpha$ to be the planar lattice generated by the matrix
	$$M_{\Lambda_h^\alpha}:=\left[
	\begin{array}{cc}
		1 & \alpha \\
		0 & \overline{\alpha}
	\end{array}
	\right].$$
\end{definition} 
It follows from the definition that $\Lambda_h^{\frac{1}{2}}=\Lambda_h$, the hexagonal lattice, and $\Lambda_h^0=\Z^2$, the orthogonal lattice. For other values of $\alpha$, $\Lambda_h^\alpha$ is a planar GWR lattice. The volume of $\Lambda_h^\alpha$ is given by $\vol{\Lambda_h^\alpha}=|\det(M_{\Lambda_h^\alpha})|=\overline{\alpha},$ and the center density is given by
$\delta(\Lambda_h^\alpha)=\frac{\lambda_1(\Lambda_h^\alpha)^2}{2^2 \vol{\Lambda_h^\alpha}}=\frac{1}{4\overline{\alpha}}.$ In particular, $\alpha\mapsto\delta(\Lambda_h^\alpha)$ is an increasing function on $\left[0,\frac{1}{2}\right]$, as can be seen from Figure \ref{fig:deformed_hexagonal}.

\begin{figure}[H]
	\centering
	\includegraphics[width=0.8\textwidth]{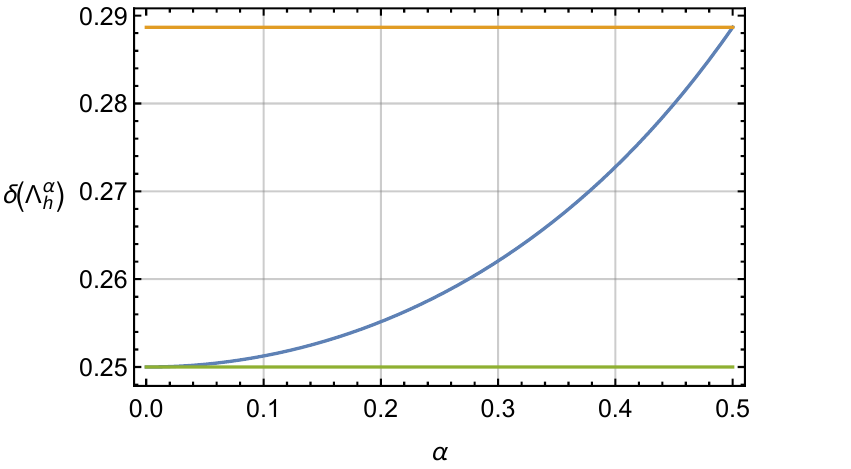}
	\caption{Center density of the deformed hexagonal lattice $\Lambda_h^\alpha$ as a function of $\alpha$. The bottom line shows $\delta(\Z^2)$ and the upper line $\delta(\Lambda_h)$.}
	\label{fig:deformed_hexagonal}
\end{figure}

\subsection{Deformed $D_n$}
\label{d_n}

The Checkerboard lattice, which is defined as $$D_n:=\set{(x_1,\dots,x_n)^T\in\Z^n:\sum_{i=1}^n x_i\equiv 0\pmod{2}},$$
is a root lattice and the densest lattice packing in dimensions 3--5. $D_3$ is also known as the face-centered cubic lattice. It is easy to see that $D_n$ has kissing number $2n(n-1)$; indeed, shortest vectors are given by permutations of $(\pm1,\pm1,0,\dots,0)$.

Using the same strategy as with the hexagonal lattice, we distort the basis vectors of $D_n$ to produce a family of GWR lattices $\set{D_n^\alpha}$, where $\alpha\in(1,\sqrt2]$, such that $\delta(D_n^\alpha)\to\delta(D_n)$ as $\alpha\to1$.

\begin{definition}
 \label{Dn_deformed}
 Let $1\leq \alpha\leq \sqrt2$ and $\overline{\alpha}:=\sqrt{2-\alpha^2}$. We define  $D_n^{\alpha}$, $n\geq3$, to be the rank $n$ lattice with generator matrix
 $$
 M_{D_n^{\alpha}}:=\left[
\begin{array}{rrrrrrrr}
 \alpha & 0 & \overline{\alpha}  & 0 & 0 & 0 & \ldots & 0 \\
 \overline{\alpha}  & \alpha & 0 & 0 & 0 & 0 &  \ldots & 0 \\
 0 & \overline{\alpha} & \alpha & \overline{\alpha} & 0 & 0&  \ldots & 0 \\
 0 & 0 & 0 & -\alpha & \overline{\alpha} & 0 & \ldots&  0  \\
 0 & 0 & 0 & 0&  -\alpha & \overline{\alpha} & \ddots & 0 \\
  \vdots & \vdots &\vdots & \vdots& \ddots & \ddots &\ddots & \vdots \\
  0 & 0 & 0 & 0& 0 & 0 & -\alpha&  \overline{\alpha} \\
  0 & 0 & 0 & 0& 0 & 0 & 0&  -\alpha \\
\end{array}
\right].
 $$
\end{definition}

 Remark that $0\leq\overline{\alpha}\leq1$ when $1\leq\alpha\leq\sqrt2$ and $\alpha^2+\overline{\alpha}^2=2$. Further, $D_n^1=D_n$ (motivating the notation) and $D_n^{\sqrt2} =\sqrt2\Z^n$. For any $\alpha\in(1,\sqrt2]$ we get a GWR lattice $D_n^\alpha$ (Theorem \ref{Dn_deformed:gwr}) and further, $\alpha\mapsto\delta(D_n^\alpha)$ is strictly decreasing on $[1,\sqrt2]$ (Proposition \ref{Dn_deformed:cd_bounds}). We start by deriving an expression for the volume of $D_n^\alpha$.

\begin{proposition}
	\label{Dn_deformed:volume}
	The volume of $D_n^{\alpha}$ is given by
	$$\vol{D_n^\alpha}=\alpha^{n-3}(\alpha^3+\overline{\alpha}^3).$$
\end{proposition}

\begin{proof}
Notice that $M_{D_n^\alpha}$ has the form of an upper triangular block matrix. Therefore, the determinant is equal to the product of the determinants of the diagonal blocks:
$$\det(M_{D_n}^\alpha)=
\left|
\begin{array}{ccc}
	\alpha & 0 & \overline{\alpha}  \\
	\overline{\alpha}  & \alpha & 0 \\
	0 & \overline{\alpha} & \alpha \\
\end{array}
\right|\cdot(-\alpha)^{n-3}=(-1)^{n-3}(\alpha^3+\overline{\alpha}^3)\alpha^{n-3}.$$
The claim follows from $\vol{D_n^\alpha}=|\det(M_{D_n^\alpha})|$.
\end{proof}

\begin{theorem}
\label{Dn_deformed:gwr}
	Let $1<\alpha\leq\sqrt2$. Then $D_n^\alpha$ has the set of minimal vectors $S(D_n^\alpha)=\set{\pm\B{b}_1,\dots,\pm\B{b}_n}$ where $\B{b}_i$ is the $i$-th column of $M_{D_n^\alpha}$. In particular, $D_n^\alpha$ is GWR and $\lambda_1^2(D_n^\alpha)=2$.
\end{theorem}

\begin{proof}
 See Appendix, Section \ref{Thm28Pf}.
\end{proof}

Knowing the volume and shortest vector length of $D_n^\alpha$, we are able to compute its center density.

\begin{corollary}
	\label{Dn_deformed:cd}
	The center density of $D_n^\alpha$ is given by
	$$\delta(D_n^\alpha)=\frac{1}{2^{n/2}\alpha^{n-3}(\alpha^3+\overline{\alpha}^3)}.$$
\end{corollary}

\begin{proof}
By Theorem \ref{Dn_deformed:gwr}, $\lambda_1^2(D_n^\alpha)=2$. Thus, using Proposition \ref{Dn_deformed:volume},
$$\delta(D_n^\alpha)=\frac{\lambda_1(D_n^\alpha)^n}{2^n\vol{D_n^\alpha}}=\frac{1}{2^{n/2}\alpha^{n-3}(\alpha^3+\overline{\alpha}^3)}.$$ \
\end{proof}

Next we will show that the center density $\delta(D_n^\alpha)$ is a strictly decreasing function on $[1,\sqrt2]$. This implies that no two lattices in the family $\set{D_n^\alpha}$ are similar. Lemma \ref{lem:Dn_cd} will be used to prove this fact.

\begin{lemma}
	\label{lem:Dn_cd}
	Let $n\geq3$ be an integer. Define the real-valued function $$f:[1,\sqrt2]\rightarrow\R,\quad f(x)=x^{n-3}(x^3+(2-x^2)^{3/2}).$$ 
	Then $f$ is strictly increasing on $[1,\sqrt2]$.
\end{lemma}

\begin{proof}
A direct computation shows that
$$f'(x)=x^{n-4} \left((n-3) \left(x^3+\left(2-x^2\right)^{3/2}\right)+3 \left(x-\sqrt{2-x^2}\right) x^2\right)>0$$
for all $x\in (1,\sqrt2)$, since $(n-3)(x^3+(2-x^2)^{3/2})\geq0$ and $x-\sqrt{2-x^2}>0$ when $x\in(1,\sqrt2)$.
\end{proof}

\begin{proposition}
\label{Dn_deformed:cd_bounds}
	The center density of $D_n^\alpha$ satisfies $\delta(\Z^n)\leq \delta(D_n^\alpha)\leq \delta(D_n),$
	and the upper bound is achieved when $\alpha=1$ and the lower bound is achieved when $\alpha=\sqrt2$. Moreover, $\delta(D_n^\alpha)$ is strictly decreasing on $[1,\sqrt2]$.
\end{proposition}

\begin{proof}
Let $f$ be the function defined in Lemma \ref{lem:Dn_cd}. Then $\delta(D_n^\alpha)=\frac{1}{2^{n/2}f(\alpha)}$ and in particular, by the previous lemma, $\delta(D_n^\alpha)$ is strictly decreasing on $[1,\sqrt2]$. Therefore,
\begin{align*}
\delta(D_n^\alpha)&\leq \frac{1}{2^{n/2}f(1)}=\frac{1}{2^{n/2+1}}=\delta(D_n), \\ 
\delta(D_n^\alpha)&\geq \frac{1}{2^{n/2}f(\sqrt2)}=\frac{1}{2^n}=\delta(\Z^n).
\end{align*}
\end{proof}
\subsubsection{Integral deformed $D_n$}
\label{sec:integral1}

From a computational perspective, the lattices $D_n^\alpha$ can be problematic, since the basis vectors might contain irrational entries, or entries close to 1 if we wanted to have a high center density. Moreover, in the context of lattice coset coding, one often looks for GWR lattices which are sublattices of the orthogonal lattice $\Z^n$. Since scaling does not change the center density of a lattice, and GWR lattices are closed under scaling, the aforementioned considerations motivate finding lattices $cD_n^\alpha\subseteq\Z^n$, where $c$ is a constant.

One way to obtain a lattice $c D_n^\alpha\subseteq\Z^n$ is to ensure that $\alpha,\overline{\alpha}\in\Q$ and then scale $D_n^\alpha$ with the denominator of $\alpha$. If we suppose that $\alpha=p/q$ for some positive and relatively prime $p,q\in\Z$ such that $1\leq p/q<\sqrt{2}$, then $\overline{\alpha}=\frac{\sqrt{2q^2-p^2}}{q}\in\Q$ if and only if $2q^2-p^2$ is a square. Equivalently, $(p,q)$ is a solution to the generalized Pell's equation $2y^2-x^2=d^2$ for some $d\in \Z$. There exists algorithms for finding a solution to such an equation (in fact, infinitely many solutions), if a solution exists. We are particularly interested in solutions for which $q$ is small, since this gives a small determinant for $q D_n^\alpha$, and for which $p/q$ is close to 1, since this gives a high center density. Thus, it suffices to check case by case all small pairs of integers $(p,q)$ which have the desired properties. Table \ref{table:pq_vals} shows some pairs of integers $(p,q)$, and corresponding $d$, $\alpha=p/q$, center density and normalized squared lattice minimum, such that $2q^2-p^2=d^2$ for some $d\in\Z$ and consequently, $qD_n^\alpha\subseteq\Z^n$. We have excluded the trivial case $p=q$ which yields $D_n$.

\begin{table}
	\centering
	\renewcommand{\arraystretch}{1.5}
	\begin{tabular}{cccccc} 
		\toprule
		$p$ & $q$ & $d$ & $\alpha$ & $\delta(D_n^\alpha)$& $\lambda_1^2({D_n^{\alpha}}')$ \\ 
		\midrule
		7 & 5 & 1 & 1.4 & $\frac{2^{-\frac{n}{2}-3}\cdot 5^n\cdot7^{3-n}}{43}$ & $4\cdot \left(\frac{2^{-\frac{n}{2}-3}\cdot 5^n\cdot7^{3-n}}{43}\right)^{\frac{2}{n}}$\\ 
		17 & 13 & 7 & 1.30769 & $\frac{2^{-\frac{n}{2}-3}\cdot 13^n\cdot17^{3-n}}{657}$ & $4\cdot \left(\frac{2^{-\frac{n}{2}-3}\cdot 13^n\cdot17^{3-n}}{657}\right)^{\frac{2}{n}}$ \\
		31 & 25 & 17 & 1.24 & $\frac{2^{-\frac{n}{2}-4}\cdot 25^n \cdot31^{3-n}}{2169}$ & $4\cdot \left(\frac{2^{-\frac{n}{2}-4}\cdot 25^n \cdot31^{3-n}}{2169}\right)^{\frac{2}{n}}$ \\
		49 & 41 & 31 & 1.19512 & $\frac{2^{-\frac{n}{2}-4} \cdot41^n \cdot49^{3-n}}{9215}$ & $4\cdot \left(\frac{2^{-\frac{n}{2}-4} \cdot41^n \cdot49^{3-n}}{9215}\right)^{\frac{2}{n}}$ \\
		71 & 61 & 49 & 1.16393 & $\frac{2^{-\frac{n}{2}-3}\cdot61^n\cdot71^{3-n}}{59445}$ & $4\cdot \left(\frac{2^{-\frac{n}{2}-3}\cdot61^n\cdot71^{3-n}}{59445}\right)^{\frac{2}{n}}$ \\
		97 & 85 & 71 & 1.14118 & $\frac{2^{-\frac{n}{2}-3}\cdot85^n\cdot97^{3-n}}{158823}$ & $4\cdot \left(\frac{2^{-\frac{n}{2}-3}\cdot85^n\cdot97^{3-n}}{158823}\right)^{\frac{2}{n}}$\\
		127 & 113 & 97 & 1.12389 & $\frac{2^{-\frac{n}{2}-5}\cdot113^n\cdot127^{3-n}}{92533}$ & $4\cdot \left(\frac{2^{-\frac{n}{2}-5}\cdot113^n\cdot127^{3-n}}{92533}\right)^{\frac{2}{n}}$ \\
		161 & 145 & 127 & 1.11034 & $\frac{2^{-\frac{n}{2}-5}\cdot145^n\cdot161^{3-n}}{194427}$ & $4\cdot \left(\frac{2^{-\frac{n}{2}-5}\cdot145^n\cdot161^{3-n}}{194427}\right)^{\frac{2}{n}}$ \\
		199 & 181 & 161 & 1.09945 & $\frac{2^{-\frac{n}{2}-3}\cdot181^n\cdot199^{3-n}}{1506735}$ & $4\cdot \left(\frac{2^{-\frac{n}{2}-3}\cdot181^n\cdot199^{3-n}}{1506735}\right)^{\frac{2}{n}}$\\
		287 & 265 & 241 & 1.08302 & $\frac{2^{-\frac{n}{2}-4}\cdot265^n\cdot287^{3-n}}{2352339}$ & $4\cdot \left(\frac{2^{-\frac{n}{2}-4}\cdot265^n\cdot287^{3-n}}{2352339}\right)^{\frac{2}{n}}$ \\
		391 & 365 & 337 & 1.07123 & $\frac{2^{-\frac{n}{2}-3}\cdot365^n\cdot391^{3-n}}{12256153}$ & $4\cdot \left(\frac{2^{-\frac{n}{2}-3}\cdot365^n\cdot391^{3-n}}{12256153}\right)^{\frac{2}{n}}$\\
		511 & 481 & 449 & 1.06237 & $\frac{2^{-\frac{n}{2}-6}\cdot481^n\cdot511^{3-n}}{3499245}$ & $4\cdot \left(\frac{2^{-\frac{n}{2}-6}\cdot481^n\cdot511^{3-n}}{3499245}\right)^{\frac{2}{n}}$ \\
		\bottomrule
	\end{tabular}
	\caption{Pairs of integers $(p,q)$ which produce a sublattice $q D_n^\alpha\subseteq \Z^n$. Here ${D_n^{\alpha}}'$ denotes the lattice $c D_n^\alpha$ where $c$ is chosen such that $\vol{cD_n^\alpha}=1$.}
	\label{table:pq_vals}
\end{table}

We provide an example which illustrates how an integral scaled $D_n^\alpha$ lattice can be obtained.

\begin{example}
	Suppose that $n=4$ and $(p,q)=(7,5)$. Then $\alpha=\frac{p}{q}=\frac{7}{5}$ and $\overline{\alpha}=\frac{d}{q}=\frac{1}{5}$, as seen from Table \ref{table:pq_vals}. In this case, a generator matrix for $qD_n^\alpha$ is given by
	$$M_{qD_n^\alpha}=\left[
	\begin{array}{cccc}
		7 & 0 & 1 & 0 \\
		1 & 7 & 0 & 0 \\
		0 & 1 & 7 & 1 \\
		0 & 0 & 0 & -7 \\
	\end{array}
	\right].$$
	By Corollary \ref{Dn_deformed:cd}, (or from Table \ref{table:pq_vals}), $$\delta(D_n^\alpha)=\frac{1}{2^{n/2}\alpha^{n-3}(\alpha^3+\overline{\alpha}^3)}=\frac{1}{2^{4/2}\cdot\left(\frac{7}{5}\right)^{4-3}\cdot\left(\left(\frac{7}{5}\right)^3+\left(\frac{1}{5}\right)^3\right)}\approx 0.0648879.$$
\end{example}

\subsection{Deformed $E_8$}
\label{e_8}

Recall the definition of the $E_8$ lattice (Definition \ref{def:gamma}), the densest lattice packing in dimension $8$, with a center density of 0.0625. Our goal in this section is to deform the basis vectors of the lattice $\Gamma_8'$ (the odd coordinate system version of $E_8$) in the same way that we did with the $D_n$ lattice, to obtain GWR lattices with good packing density in dimension 8.

\begin{definition}
	Let $1\leq \alpha\leq \sqrt2$ and $\overline{\alpha}:=\sqrt{2-\alpha^2}$. We define $E_8^\alpha$ to be the rank $8$ lattice with generator matrix
	$$M_{E_8^\alpha}:=\frac{1}{2}
	\left[
	\begin{array}{rrrrrrrr}
		1 & 2\alpha  & 0 & 0 & 0 & 0 & 0 & 0 \\
		1 & 2\overline{\alpha } & 2\alpha  & 0 & 0 & 0 & 0 & 0 \\
		1 & 0 & 2\overline{\alpha } & 2\alpha  & 0 & 0 & 0 & 0 \\
		1 & 0 & 0 & 2\overline{\alpha } & 2\alpha  & 0 & 0 & 0 \\
		1 & 0 & 0 & 0 & 2\overline{\alpha } & 2\alpha  & 0 & 0 \\
		1 & 0 & 0 & 0 & 0 & 2\overline{\alpha } & 2\alpha  & 0 \\
		-1 & 0 & 0 & 0 & 0 & 0 & 2\overline{\alpha } & 2\alpha  \\
		1 & 0 & 0 & 0 & 0 & 0 & 0 & 2\overline{\alpha } \\
	\end{array}
	\right].
	$$
\end{definition}

Note that $\alpha^2+\overline{\alpha}^2=2$ and that $E_8^1=\Gamma_8'$. This motivates the notation.

\begin{lemma}
	\label{lem:e8_vol}
	The volume of $E_8^\alpha$ is given by
	$$\vol{E_8^\alpha}=\frac{1}{2}\left( \overline{\alpha}^2(\alpha-\overline{\alpha})(\alpha^4+\alpha^2\overline{\alpha}^2+\overline{\alpha}^4)+\alpha^6(\alpha+\overline{\alpha})\right).$$
\end{lemma}

\begin{proof}
A direct computation shows that $$\det(M_{E_8^\alpha})=-\frac{1}{2}\left( \overline{\alpha}^2(\alpha-\overline{\alpha})(\alpha^4+\alpha^2\overline{\alpha}^2+\overline{\alpha}^4)+\alpha^6(\alpha+\overline{\alpha})\right).$$
The claim follows from $\vol{E_8^\alpha}=|\det(M_{E_8}^\alpha)|$.
\end{proof}

The following theorem states that $E_8^\alpha$ is GWR for all $1<\alpha\leq\sqrt2$. This means that, in particular, $\set{E_8^\alpha}$ is a family of GWR lattices with center density approaching $\delta_8=0.0625$ as $\alpha\to1$.

\begin{theorem}
	\label{conj}
	Let $1<\alpha\leq \sqrt2$. Then $E_8^\alpha$ has the set of minimal vectors $S(E_8^\alpha)=\set{\pm\B{b}_1,\dots,\pm\B{b}_8}$, where $\B{b}_i$ is the $i$-th column of $M_{E_8^\alpha}$. In particular, $E_8^\alpha$ is GWR and $\lambda_1^2(E_8^\alpha)=2$.
\end{theorem}

\begin{proof}
 The proof is similar to the proof of Theorem \ref{Dn_deformed:gwr} (see Appendix, Section \ref{Thm28Pf}), but even longer and relies on a computer program so we simply omit it here.
\end{proof}

We can now compute the center density of the lattice $E_8^\alpha$.

\begin{proposition}
	The center density of $E_8^\alpha$ is given by
	$$\delta(E_8^\alpha)=\frac{1}{8 \left( \overline{\alpha}^2(\alpha-\overline{\alpha})(\alpha^4+\alpha^2\overline{\alpha}^2+\overline{\alpha}^4)+\alpha^6(\alpha+\overline{\alpha})\right)}.$$
\end{proposition}
\begin{proof}
By Lemma \ref{lem:e8_vol} and Theorem \ref{conj}, \begin{align*}\delta(E_8^\alpha)&=\frac{\lambda_1(E_8^\alpha)^8}{2^8\vol{E_8^\alpha}}\\
	&
	=\frac{2^4}{2^8\cdot \frac{1}{2}\left( \overline{\alpha}^2(\alpha-\overline{\alpha})(\alpha^4+\alpha^2\overline{\alpha}^2+\overline{\alpha}^4)+\alpha^6(\alpha+\overline{\alpha})\right)} \\
	&=\frac{1}{8 \left( \overline{\alpha}^2(\alpha-\overline{\alpha})(\alpha^4+\alpha^2\overline{\alpha}^2+\overline{\alpha}^4)+\alpha^6(\alpha+\overline{\alpha})\right)}.
\end{align*}
\end{proof}

\subsubsection{Integral deformed $E_8$ lattice}
\label{sec:integral2}

Our goal in this section is to find scaled variants of the $E_8^\alpha$ lattice as a sublattice of $\Z^8$. We apply the exact same strategy as with the $D_n^\alpha$ lattice. Let $\alpha=\frac{p}{q}$ for some relatively prime, positive $p,q\in\Z$ such that $1\leq p/q<\sqrt2$. Then $\overline{\alpha}=\frac{\sqrt{2q^2-p^2}}{q}\in\Q$ if and only if $2q^2-p^2=d^2$ for some $d\in\Z$. If this is the case, then $2qE_8^\alpha$ is a sublattice of $\Z^8$. We want to find small values for $q$ to get a small determinant for $2qE_8^\alpha$, and a value of $\alpha=\frac{p}{q}$ which is close to 1 to maximize the center density. Table \ref{table:pq_vals2} shows some pairs of integers $(p,q)$, and corresponding $\alpha$, $d$, center density and squared lattice minimum, such that $2q^2-p^2=d^2$ for some $d\in\Z$ and thus, $2q{E_8}^\alpha\subseteq\Z^n$. As the table indicates, if we desire a high center density, we have to accept large values for $q$.

\begin{table}[H]
	\centering
	\begin{tabular}{cccccc} 
		\toprule
		$p$ & $q$ & $d$ & $\alpha$ & $\delta(E_8^\alpha)$ & $\lambda_1^2({E_8^\alpha}')$ \\ 
		\midrule
		7 & 5 & 1 & 1.4 & 0.0102162 & 1.27169 \\
		17 & 13 & 7 & 1.30769 & 0.0124829 & 1.33702 \\
		31 & 25 & 17 & 1.24 & 0.0159616 & 1.42177 \\
		49 & 41 & 31 & 1.19512 & 0.0192763 & 1.49045 \\
		71 & 61 & 49 & 1.16393 & 0.0222471 & 1.54482 \\
		97 & 85 & 71 & 1.14118 & 0.0248757 & 1.58856 \\
		127 & 113 & 97 & 1.12389 & 0.0272007 & 1.62445 \\
		161 & 145 & 127 & 1.11034 & 0.0292647 & 1.65442 \\
		241 & 221 & 199 & 1.0905 & 0.0327571 & 1.70171 \\
		337 & 313 & 287 & 1.07668 & 0.0355924 & 1.7374 \\
		449 & 421 & 391 & 1.06651 & 0.0379372 & 1.76533 \\
		647 & 613 & 577 & 1.05546 & 0.0407789 & 1.7975 \\
		881 & 841 & 799 & 1.04756 & 0.0430324 & 1.82183 \\
		1249 & 1201 & 1151 & 1.03997 & 0.0453987 & 1.84638 \\
		1799 & 1741 & 1681 & 1.03331 & 0.0476548 & 1.8689 \\
		2591 & 2521 & 2449 & 1.02777 & 0.0496839 & 1.88849 \\
		4049 & 3961 & 3871 & 1.02222 & 0.0518646 & 1.90888 \\
		6727 & 6613 & 6497 & 1.01724 & 0.0539629 & 1.9279 \\
		30257 & 30013 & 29767 & 1.00813 & 0.0582025 & 1.9647 \\
		95047 & 94613 & 94177 & 1.00459 & 0.0600098 & 1.97977 \\
		301087 & 300313 & 299537 & 1.00258 & 0.0610791 & 1.98853 \\
		\bottomrule
	\end{tabular}
	\caption{Examples of pairs of integers $(p,q)$ which produce a sublattice $2q E_8^\alpha\subseteq \Z^8$. Here ${E_8^\alpha}'$ denotes the lattice $c E_8^{\alpha}$ where $c$ is chosen such that $\vol{c E_8^\alpha}=1$.}
	\label{table:pq_vals2}
\end{table}

\section{Conclusions and future work}
\label{sec:conclusion}





Motivated by the attempt to construct lattices suitable for lattice coset codes for the fading wiretap channel as well as by the mathematical question itself, we have in this paper explored different constructions of generic well-rounded lattices: sublattices of tame lattices and deformed dense lattices. The task of finding GWR lattices with good sphere packing density was motivated by the findings in \emph{e.g.} \cite{analysis,Gnilke-Barreal,Gnilke}, \cite{towards} where it was demonstrated, using fading wiretap channel simulations, that having a small kissing number and a high packing density for the eavesdropper's lattice can lower the correct decoding probability of the eavesdropper.

Furthermore, as discussed in Sec. \ref{sec:app}, when the communication channel is of a good quality, the performance of a lattice code is dictated by its diversity and minimum product distance, which both should be maximized.  The deformed lattices $D_n^\alpha$ and $E_8^\alpha$ defined in this paper are clearly not full diversity. However, a natural way to obtain full diversity variants of the integral deformed lattices discussed in Sections \ref{sec:integral1} and \ref{sec:integral2} is to apply orthogonal transformations which maximize modulation diversity and minimum product distance for the lattice $\Z^n$. Then, since the integral deformed lattices are sublattices of $\Z^n$, if we apply the same orthogonal transformation, the obtained lattices must also have full diversity and minimum product distance lower bounded by the minimum product distance of the optimally rotated $\Z^n$ lattice. For some currently best known rotations,   maximizing the minimum product distance of $\Z^n$, see  \cite{Viterbo_tables}. Let us also note that in scenarios  where the channel quality is constantly low (\emph{e.g.}, due to long communication distance, high speed, physical obstacles, or low transmission power), non-full-diversity lattices can be of great interest. 

As for more applied future work, it is left to evaluate how well the constructed lattices perform in actual wiretap channel simulations. As mentioned earlier, one design criterion for the fading wiretap channel is to minimize the flatness factor of the sublattice, which is equivalent to minimizing the theta series. However, at least in dimension 8, the deformed lattices $E_8^\alpha$ cannot be minimizers of the theta function, since $E_8$ is the unique minimizer. Still, the theta series of $E_8^\alpha$ approaches that of the $E_8$ lattice as the parameter $\alpha$ approaches~1. 

In this paper, we also demonstrated the fact that, at least in dimensions less than~5 and in dimension 8, there exist GWR lattices with density arbitrarily close to the optimal density. One would expect this to hold for other dimensions as well. As far as we know, the research on GWR lattices is rather scarce, even though GWR lattices are widely represented in the set of well-rounded lattices. For instance, in dimension~1, all WR lattices are GWR, and in dimension 2, all WR lattices but the hexagonal lattice are GWR. 

\bibliographystyle{siamplain}

\begin{thebibliography}{10}

\bibitem{Alves}
{\sc C.~Alves, W.~L. da~Silva~Pinto, and A.~A. de~Andrade}, {\em Well-rounded
  lattices via polynomials}, CoRR, abs/1904.03510 (2019),
  \url{http://arxiv.org/abs/1904.03510}, \href{http://arxiv.org/abs/1904.03510}
  {arXiv:1904.03510}.

\bibitem{bacher2015constructions}
{\sc R.~Bacher}, {\em Constructions of some perfect integral lattices with
  minimum $4$}, Journal de th{\'e}orie des nombres de Bordeaux, 27 (2015),
  pp.~655--687.

\bibitem{Amaro_approx}
{\sc A.~Barreal, M.~T. Damir, R.~Freij-Hollanti, and C.~Hollanti}, {\em An
  approximation of theta functions with applications to communications}, SIAM
  Journal on Applied Algebra and Geometry, 4 (2020), pp.~471--501.

\bibitem{belfiore}
{\sc J.-C. {Belfiore} and F.~{Oggier}}, {\em Lattice code design for the
  {Rayleigh} fading wiretap channel}, in 2011 IEEE International Conference on
  Communications Workshops (ICC), 2011, pp.~1--5.

\bibitem{Belfiore2010Secrecy-gain}
{\sc J.-C. Belfiore and F.~E. Oggier}, {\em Secrecy gain: A wiretap lattice
  code design}, 2010 International Symposium On Information Theory \& Its
  Applications,  (2010), pp.~174--178.

\bibitem{blanc2015}
{\sc X.~Blanc and M.~Lewin}, {\em The crystallization conjecture: A review},
  EMS Surveys in Mathematical Sciences, EMS, 2 (2015), pp.~255--306.

\bibitem{bm}
{\sc W.~Bola{\~n}os and G.~Mantilla-Soler}, {\em The trace form over cyclic
  number fields}, Canadian Journal of Mathematics,  (2020), pp.~1--23.

\bibitem{bottcher2015lattices}
{\sc A.~B\"ottcher, L.~Fukshansky, S.~R. Garcia, and H.~Maharaj}, {\em On
  lattices generated by finite abelian groups}, SIAM Journal on Discrete
  Mathematics, 29 (2015), pp.~382--404.

\bibitem{Chorti}
{\sc A.~Chorti, C.~Hollanti, J.~Belfiore, and H.~Poor}, {\em Physical layer
  security: A paradigm shift in data confidentiality}, vol.~358 of Springer
  Lecture Notes in Electrical Engineering, 2016, pp.~1--15.

\bibitem{cohn2019universal}
{\sc H.~Cohn, A.~Kumar, S.~D. Miller, D.~Radchenko, and M.~Viazovska}, {\em
  Universal optimality of the {$E_8$} and {Leech} lattices and interpolation
  formulas}, Annals of Mathematics (2), 196 (2022), pp.~983--1082.

\bibitem{cp}
{\sc P.~E. Conner and R.~Perlis}, {\em A Survey of Trace Forms of Algebraic
  Number Fields}, World Scientific, 1984.

\bibitem{conway}
{\sc J.~H. Conway and N.~J.~A. Sloane}, {\em Sphere Packings, Lattices and
  Groups (3. ed.)}, New York: Springer, 1999.

\bibitem{costa2017lattices}
{\sc S.~I. Costa, F.~Oggier, A.~Campello, J.-C. Belfiore, and E.~Viterbo}, {\em
  Lattices Applied to Coding for Reliable and Secure Communications}, Springer,
  2017.

\bibitem{coulangeon}
{\sc R.~Coulangeon}, {\em Spherical designs and zeta functions of lattices},
  International Mathematics Research Notices, 2006 (2006).

\bibitem{analysis}
{\sc M.~T. Damir, O.~Gnilke, L.~Amor\'os, and C.~Hollanti}, {\em Analysis of
  some well-rounded lattices in wiretap channels}, in 2018 IEEE 19th
  International Workshop on Signal Processing Advances in Wireless
  Communications (SPAWC), IEEE, 2018, pp.~1--5.

\bibitem{Taoufiq-Dave}
{\sc M.~T. Damir and D.~Karpuk}, {\em Well-rounded twists of ideal lattices
  from real quadratic fields}, Journal of Number Theory, 196 (2019),
  pp.~168--196.

\bibitem{towards}
{\sc M.~T. Damir, A.~Karrila, L.~Amorós, O.~W. Gnilke, D.~Karpuk, and
  C.~Hollanti}, {\em Well-rounded lattices: Towards optimal coset codes for
  {Gaussian} and fading wiretap channels}, IEEE Transactions on Information
  Theory, 67 (2021), pp.~3645--3663.

\bibitem{damir2020bases}
{\sc M.~T. Damir and G.~Mantilla-Soler}, {\em Bases of minimal vectors in tame
  lattices}, Acta Arithmetica, 205 (2022), pp.~265--285.

\bibitem{costawell}
{\sc R.~R. de~Araujo and S.~I. Costa}, {\em Well-rounded algebraic lattices in
  odd prime dimension}, Archiv der Mathematik, 112 (2019), pp.~139--148.

\bibitem{delone}
{\sc B.~N. Delone and S.~S. Ryshkov}, {\em A contribution to the theory of the
  extrema of a multi-dimensional $\zeta$-function}, in Doklady Akademii Nauk,
  vol.~173, 1967, pp.~991--994.

\bibitem{Fukshansky2}
{\sc L.~Fukshansky, G.~Henshaw, P.~Liao, M.~Prince, X.~Sun, and S.~Whitehead},
  {\em On well-rounded ideal lattices {II}}, International Journal of Number
  Theory, 9 (2013), pp.~139--154.

\bibitem{Fukshansky1}
{\sc L.~Fukshansky and K.~Petersen}, {\em On well-rounded ideal lattices},
  International Journal of Number Theory, 8 (2012), pp.~189--206.

\bibitem{Gnilke-Barreal}
{\sc O.~W. Gnilke, A.~Barreal, A.~Karrila, H.~T.~N. Tran, D.~A. Karpuk, and
  C.~Hollanti}, {\em Well-rounded lattices for coset coding in {MIMO} wiretap
  channels}, in Proc. IEEE International Telecommunication Networks and
  Applications Conference, 2016, pp.~289--294.

\bibitem{Gnilke}
{\sc O.~W. Gnilke, H.~T.~N. Tran, A.~Karrila, and C.~Hollanti}, {\em
  Well-rounded lattices for reliability and security in {R}ayleigh fading
  {SISO} channels}, in Proc. IEEE Information Theory Workshop, 2016,
  pp.~359--363.

\bibitem{Belfiore-flatness}
{\sc C.~Ling, L.~Luzzi, J.-C. Belfiore, and D.~Stehl{\'e}}, {\em Semantically
  secure lattice codes for the {G}aussian wiretap channel}, {IEEE} Transactions
  on Information Theory, 60 (2014), pp.~6399--6416.

\bibitem{luzzi_isit16}
{\sc L.~Luzzi, R.~Vehkalahti, and C.~Ling}, {\em Almost universal codes for
  fading wiretap channels}, in IEEE Int. Symp. Inf. Theory, 2016.

\bibitem{McMullenMinkowski}
{\sc C.~McMullen}, {\em Minkowski's conjecture, well-rounded lattices and
  topological dimension}, Journal of the American Mathematical Society, 18
  (2005), pp.~711--734.

\bibitem{Regev-smoothing}
{\sc D.~Micciancio and O.~Regev}, {\em Worst-case to average-case reductions
  based on {G}aussian measures}, SIAM J. Comput., 37(1) (2007), pp.~267--302.

\bibitem{Montgomery}
{\sc H.~L. Montgomery}, {\em Minimal theta functions}, Glasgow Mathematical
  Journal, 30(1) (1988), pp.~75--85.

\bibitem{nebe_sloane_2012}
{\sc G.~Nebe and N.~J.~A. Sloane}, {\em Table of densest packings presently
  known}.
\newblock
  \url{http://www.math.rwth-aachen.de/~Gabriele.Nebe/LATTICES/density.html},
  2012.
\newblock [Online].

\bibitem{Oggier-Sole-Belfiore}
{\sc F.~{Oggier}, P.~{Sol\'e}, and J.~{Belfiore}}, {\em Lattice codes for the
  wiretap {G}aussian channel: Construction and analysis}, 62 (2016),
  pp.~5690--5708.

\bibitem{viterbo}
{\sc F.~Oggier and E.~Viterbo}, {\em Algebraic Number Theory and Code Design
  for {Rayleigh} Fading Channels}, vol.~1 of Foundations and Trends in
  Communications and Information Theory, Now Publisher Inc., 2004.

\bibitem{Wyner-Ozarow}
{\sc L.~H. Ozarow and A.~D. Wyner}, {\em Wire-tap channel {II}}, AT\&T Bell
  Laboratories technical journal, 63 (1984), pp.~2135--2157.

\bibitem{Sarnak-Strombergsson}
{\sc P.~Sarnak and A.~Str{\"o}mbergsson}, {\em Minima of {E}pstein’s zeta
  function and heights of flat tori}, Inventiones mathematicae, 165 (2006),
  pp.~115--151.

\bibitem{Viterbo_tables}
{\sc E.~Viterbo and F.~E. Oggier}, {\em Full diversity rotations}.
\newblock \url{
  http://www.ecse.monash.edu.au/staff/eviterbo/rotations/rotations.html}, 2005.
\newblock [Online].

\bibitem{wyner}
{\sc A.~D. Wyner}, {\em The wiretap channel}, Bell system technical journal, 54
  (1975), pp.~1355--1387.

\end{thebibliography}

\section*{Appendix} 

\subsection{Proof of Theorem \ref{Dn_deformed:gwr}}
\label{Thm28Pf}

\begin{proof}
The case $\alpha=\sqrt2$ is clear since $D_n^{\sqrt2}=\sqrt2\Z^n$. Let us therefore assume $1<\alpha<\sqrt2$. Note that $\norm{\B{b}_i}^2=\alpha^2+\overline{\alpha}^2=2$ for all $1\leq i\leq n$, which gives $\lambda_1^2(D_n^\alpha)\leq2$. Now suppose that $\B{x}\in S(D_n^\alpha)$; then $\norm{\B{x}}^2\leq2$. Write $\B{x}=\sum_{i=1}^n c_i\B{b}_i$, where $c_i\in\Z$ and not all $c_i$'s are zero. Then
$$\B{x}=
\left[
\begin{array}{c}
	x_1\\
	x_2\\
	x_3\\
	x_4\\
	x_5\\
	\vdots\\
	x_{n-1}\\
	x_n
\end{array}
\right]=
\left[
\begin{array}{c}
	c_1\alpha+c_3\overline{\alpha} \\
	c_1\overline{\alpha}+c_2\alpha \\
	c_2\overline{\alpha}+c_3\alpha+c_4\overline{\alpha} \\
	-c_4\alpha+c_5\overline{\alpha} \\
	-c_5\alpha+c_6\overline{\alpha} \\
	\vdots \\
	-c_{n-1}\alpha+c_n\overline{\alpha} \\
	-c_n\alpha
\end{array}
\right].$$
We have that $\norm{\B{x}}^2\geq x_n^2=c_n^2\alpha^2> c_n^2$ and hence $|c_n|\leq1$. If $c_n=0$, then $\norm{\B{x}}^2\geq x_{n-1}^2=c_{n-1}^2\alpha^2>c_{n-1}^2$ and so $|c_{n-1}|\leq1$. Otherwise, if $|c_n|=1$, then also $|c_{n-1}|\leq1$. To see this, suppose that $|c_{n-1}|\geq2$. Then
$\norm{\B{x}}^2\geq x_{n-1}^2+x_n^2\geq (2\alpha-\overline{\alpha})^2+\alpha^2>\alpha^2+\alpha^2>2$, a contradiction.
Inductively, we conclude $|c_i|\leq1$ for $i=4,5,\dots,n$. 

Consider the coefficient $c_4$. We may without loss of generality assume that $c_4\in\set{0,1}$, since we can always replace $\B{x}$ by $-\B{x}$.

\textbf{Case 1}: $c_4=0$. We have $$\norm{\B{x}}^2\geq x_1^2+x_2^2+x_3^2=(c_1\alpha+c_3\overline{\alpha})^2+(c_1\overline{\alpha}+c_2\alpha)^2+(c_2\overline{\alpha}+c_3\alpha)^2.$$
Note that $c_1,c_2,c_3$ cannot all be non-zero. If this was the case, then two of them, say $c_i$ and $c_j$, would have the same sign and we would get the contradiction $\norm{\B{x}}^2\geq(c_i\alpha+c_j\overline{\alpha})^2\geq (\alpha+\overline{\alpha})^2=2+2\alpha\overline{\alpha}>2$. Now suppose that $c_i=0$ for some $i\in\set{1,2,3}$ and denote by $c_j$ and $c_k$ the other two coefficients. Then if both $c_j$ and $c_k$ are non-zero, we get
$$\norm{\B{x}}^2\geq c_j^2\alpha^2+c_k^2\overline{\alpha}^2+(c_k\alpha+c_j\overline{\alpha})^2>\alpha^2+\overline{\alpha}^2=2,$$
since $(c_k\alpha+c_j\overline{\alpha})^2>0$ when $c_j,c_k\in\set{\pm1}$. This is a contradiction. We conclude that either \begin{enumerate}[label=(\roman*)]
	\item $|c_i|=1$ for some $i\in\set{1,2,3}$ and $c_j=0$ for $j\in\set{1,2,3}\setminus{\set{i}}$ or,
	\item $c_i=0$ for all $i\in\set{1,2,3}$.
\end{enumerate}
In case (i), 
\begin{align}
	\norm{\B{x}}^2&=\alpha^2+\overline{\alpha}^2+x_4^2+\cdots+x_n^2 \\
	&=2+c_5^2\overline{\alpha}^2+(-c_5\alpha+c_6\overline{\alpha})^2+\dots+(-c_{n-1}\alpha+c_n\overline{\alpha})^2+c_n^2\alpha^2
\end{align} and thus $c_5=\cdots=c_n=0$. In case (ii),
$$\norm{\B{x}}^2= x_4^2+\dots+x_n^2=c_5^2\overline{\alpha}^2+(-c_5\alpha+c_6\overline{\alpha})^2+\dots+(-c_{n-1}\alpha+c_n\overline{\alpha})^2+c_n^2\alpha^2.$$ Suppose that there are at least two non-zero coefficients $c_i$ where $i\in\set{5,\dots,n}$. Let $c_j$ and $c_k$ be two such coefficients, with $j$ minimal and $k$ maximal. Then $$\norm{\B{x}}^2=c_j^2\overline{\alpha}^2+(-c_j\alpha+c_{j+1}\overline{\alpha})^2+\cdots+(-c_{k-1}\alpha+c_k\overline{\alpha})^2+c_k^2\alpha^2> c_j^2\overline{\alpha}^2+c_k^2\alpha^2=2,$$ a contradiction. We conclude that $|c_i|=1$ for one $i\in\set{5,\dots,n}$ and $c_i=0$ else. \\

\textbf{Case 2}: $c_4=1$. Note that
$$x_1^2+x_2^2+x_3^2=(c_1\alpha+c_3\overline{\alpha})^2+(c_1\overline{\alpha}+c_2\alpha)^2+((c_2+1)\overline{\alpha}+c_3\alpha)^2$$
$$x_4^2+\cdots+x_n^2=(-\alpha+c_5\overline{\alpha})^2+\cdots+(-c_{n-1}\alpha+c_n\overline{\alpha})^2+c_n^2\alpha^2\geq\alpha^2.$$
Suppose that $c_1,c_2,c_3$ are all non-zero. If $c_1$ and $c_3$ have the same sign or $c_1$ and $c_2$ have the same sign, we get a contradiction $\norm{\B{x}}^2\geq x_1^2+x_2^2+x_3^2\geq(\alpha+\overline{\alpha})^2>2$. On the other hand, if $c_2$ and $c_3$ have the same sign, then $x_3^2\geq \alpha^2$ which implies $\norm{\B{x}}^2\geq \alpha^2+\alpha^2>2$, a contradiction. Therefore, $c_i=0$ for some $i\in\set{1,2,3}$. If $c_1=0$ then $$\norm{\B{x}}^2\geq c_3^2\overline{\alpha}^2+c_2^2\alpha^2+((c_2+1)\overline{\alpha}+c_3\alpha)^2+\alpha^2$$
which implies $c_2=0$ and $c_3=0$. If $c_2=0$ then
$$\norm{\B{x}}^2\geq (c_1\alpha+c_3\overline{\alpha})^2+c_1^2\overline{\alpha}^2+(\overline{\alpha}+c_3\alpha)^2+\alpha^2$$
which implies $c_1=0$ and $c_3=0$. If $c_3=0$ then
$$\norm{\B{x}}^2\geq c_1^2\alpha^2+(c_1\overline{\alpha}+c_2\alpha)^2+(c_2+1)^2\overline{\alpha}^2+\alpha^2$$
which implies $c_1=0$ and $c_2=0$. In any case, $c_1=c_2=c_3=0$. Therefore, $x_1^2+x_2^2+x_3^2=\overline{\alpha}^2$ and $$\norm{\B{x}}^2=\overline{\alpha}^2+(-\alpha+c_5\overline{\alpha})^2+\cdots+(-c_{n-1}\alpha+c_n\overline{\alpha})^2+c_n^2\alpha^2\geq2$$
with equality holding if and only if $c_5=c_6=\cdots=c_n=0$.

Cases 1 and 2 imply that $\B{x}=\pm\B{b}_i$ for some $i\in\set{1,\dots,n}$. This shows that $S(D_n^\alpha)=\set{\pm \B{b}_1,\dots,\pm\B{b}_n}$. To see that $D_n^\alpha$ is WR, note that $M_{D_n^\alpha}$ is non-singular and thus $S(D_n^\alpha)$ spans $\R^n$. Moreover, $D_n^\alpha$ is GWR since $\kappa(D_n^\alpha)=|S(D_n^\alpha)|=2n$.
\end{proof}

\end{document}